\documentclass{article}

\usepackage{amsmath}
\usepackage{amsthm}
\usepackage{amssymb}
\usepackage{comment}
\usepackage{tikz}
\usetikzlibrary{arrows,decorations.pathmorphing,backgrounds,positioning,fit,petri,patterns}

\usepackage[top=30truemm,bottom=30truemm,left=25truemm,right=25truemm]{geometry}

\usepackage{here}

\usepackage{enumitem}
\usepackage{todonotes}

\usepackage[nobysame,alphabetic]{amsrefs}
\usepackage{graphicx}
\usepackage{stmaryrd}
\usepackage[all]{xy}
\usepackage[vcentermath]{youngtab}
\usepackage{enumitem}

\allowdisplaybreaks[4]

\theoremstyle{definition}
\newtheorem{defi}{Definition}[section]
\newtheorem{defi-prop}[defi]{Definition-Proposition}

\theoremstyle{plain}
\newtheorem{lemm}[defi]{Lemma}

\newtheorem{prop}[defi]{Proposition}
\newtheorem{theo}[defi]{Theorem}
\newtheorem{coro}[defi]{Corollary}

\theoremstyle{remark}
\newtheorem{rema}[defi]{Remark}
\newtheorem*{rema*}{Remark}
\newtheorem{exam}[defi]{Example}
\newtheorem*{exam*}{Example}

\newcommand{\Todo}[1]{}

\newcommand{\kks}[2][k]{g^{(#1)}_{#2}}
\newcommand{\kkss}[2][k]{\widetilde{g}^{(#1)}_{#2}}
\newcommand{\wth}{\widetilde{h}}
\newcommand{\ks}[1]{s^{(k)}_{#1}}
\newcommand{\Pk}[1][k]{\mathcal{P}_{#1}}
\newcommand{\Cn}[1][k+1]{\mathcal{C}_{#1}}
\newcommand{\aSn}[1][k+1]{\ti S_{#1}}
\newcommand{\aG}[1][k+1]{\ti S_{#1}^\circ}
\newcommand{\Wo}{W^\circ}
\newcommand{\tSn}{\ti S_{k+1}/S_{k+1}}
\newcommand{\Lk}{\Lambda^{(k)}}
\newcommand{\la}{\lambda}
\newcommand{\La}{\Lambda}
\newcommand{\ka}{\kappa}

\newcommand{\Z}{\mathbb{Z}}

\newcommand{\Rt}{R_t}

\newcommand{\lra}{\longrightarrow}

\newcommand{\bdd}{\mathfrak{p}}
\newcommand{\core}{\mathfrak{c}}
\newcommand{\StoC}{\mathfrak{s}}

\newcommand{\sm}{\setminus}

\newcommand{\ti}{\tilde}

\newcommand{\hook}[2]{\mathrm{hook}_{#1}(#2)}

\renewcommand{\emptyset}{\varnothing}

\newcommand{\tikzitem}[2][1.0]{
	\begin{tikzpicture}[scale=#1]
		#2
	\end{tikzpicture}
}

\newcommand{\tikznum}[4][\small]{
	\node at (#3-0.6,#2-0.5) {#1 #4};
}

\newcommand{\tikzvnum}[3][\small]{
	\foreach \a [count=\r] in {#3}{
		\tikznum[#1]{#2}{\r}{\a};
	}
}

\newcommand{\tikzbox}[3][]{
	\draw [#1] (#3,#2) rectangle +(-1,-1);
}

\newcommand{\tikzboxnum}[4][\small]{
	\tikzbox{#2}{#3}
	\tikznum[#1]{#2}{#3}{#4}
}

\newcommand{\tikzydiag}[2][]{
	\foreach \li [count=\i] in {#2}{
		\foreach \j in {1,...,\li}{
			\tikzbox[#1]{\i}{\j};
		}
	}
}
\newcommand{\tikzyfill}[2][gray]{
	\foreach \li [count=\i] in {#2}{
		\foreach \j in {1,...,\li}{
			\fill[#1] (\j,\i) rectangle +(-1,-1);
		}
	}
}

\newcommand{\tikzydnum}[3]{
	\foreach \a [count=\r] in {#3}{
		\tikznum[#1]{#2}{\r}{\a};
	}
}

\newcommand{\sk}[3][1.0]{
	\begin{tikzpicture}[scale=#1]
		\tikzyfill{#3}
		\tikzydiag{#2}
	\end{tikzpicture}
}

\newcommand{\redex}[1]{\langle #1 \rangle}
\newcommand{\redprod}[2]{\redex{#1}\redex{#2}}
\newcommand{\redprodd}[3]{\redex{#1}\redex{#2}\redex{#3}}

\newcommand{\lecov}{<\mathrel{\mkern-5mu}\mathrel{\cdot}}
\newcommand{\gecov}{\mathrel{\cdot}\mathrel{\mkern-5mu}>}
\newcommand{\lecovL}{\lecov\hspace{-1mm}_{L}}

\newcommand{\lecovR}{\lecov\hspace{-1mm}_{R}}

\newcommand{\subsetcov}{\subset\mathrel{\mkern-5mu}\mathrel{\cdot}}

\newcommand{\SLjoin}{\vphantom{\vee}_{S}\hspace{-1mm}\vee_L}
\newcommand{\SRjoin}{\vphantom{\vee}_{S}\hspace{-1mm}\vee_R}
\newcommand{\LSjoin}{\vphantom{\vee}_{L}\hspace{-1mm}\vee_S}

\newcommand{\SLmeet}{\vphantom{\wedge}_{S}\hspace{-1mm}\wedge_L}
\newcommand{\SRmeet}{\vphantom{\wedge}_{S}\hspace{-1mm}\wedge_R}
\newcommand{\LSmeet}{\vphantom{\wedge}_{L}\hspace{-1mm}\wedge_S}
\newcommand{\RSmeet}{\vphantom{\wedge}_{R}\hspace{-0.5mm}\wedge_S}

\newcommand{\RD}{\mathrm{RD}}
\newcommand{\RI}{\mathrm{RI}}

\title{A Pieri-type formula and a factorization formula for sums of $K$-$k$-Schur functions}
\author{Motoki Takigiku}
\date{\today}

\AtEndDocument{\bigskip{\footnotesize%
  \textsc{Graduate School of Mathematical Sciences, the University of Tokyo, Japan} \par  
  \textit{E-mail address}: \texttt{takigiku@ms.u-tokyo.ac.jp} 
}}

\begin{document}

\maketitle

\begin{abstract}
We give a Pieri-type formula for the sum of $K$-$k$-Schur functions 
$\sum_{\mu\le\lambda} g^{(k)}_{\mu}$ over a principal order ideal of the 
poset of $k$-bounded partitions under the strong Bruhat order, which sum 
we denote by $\widetilde{g}^{(k)}_{\lambda}$.
As an application of this, we also give a $k$-rectangle factorization 
formula 
$\widetilde{g}^{(k)}_{R_t\cup\lambda}=\widetilde{g}^{(k)}_{R_t} \widetilde{g}^{(k)}_{\lambda}$
where $R_t=(t^{k+1-t})$, 
analogous to that of $k$-Schur functions
$s^{(k)}_{R_t\cup\lambda}=s^{(k)}_{R_t}s^{(k)}_{\lambda}$.
\end{abstract}

\tableofcontents

\section{Introduction}

Let $k$ be a positive integer.
\textit{$K$-$k$-Schur functions} $\kks{\la}$
are inhomogeneous symmetric functions parametrized by
$k$-bounded partitions $\la$, 
namely by the weakly decreasing strictly positive integer sequences
$\la=(\la_1,\dots,\la_l)$, $l\in\Z_{\ge 0}$, whose terms are all bounded by $k$.
They are $K$-theoretic analogues of another family of symmetric functions called \textit{$k$-Schur functions} $\ks{\la}$, which are homogeneous and also parametrized by $k$-bounded partitions.
The set of $k$-bounded partitions is denoted by $\Pk$.

In this paper we give a Pieri-type formula for a certain sum of $K$-$k$-Schur functions
(Theorem \ref{theo:StrongSumPieri} and \ref{theo:StrongSumPieri'})
and a factorization formula
(Theorem \ref{theo:StrongSumFactorization})
involving the \textit{$k$-rectangle partitions} $R_t$ defined later,
mainly using combinatorial properties of the strong (Bruhat) and weak orderings on the affine symmetric groups.

Historically, 
$k$-Schur functions were first introduced by Lascoux, Lapointe and Morse \cite{MR1950481},
and subsequent studies led to several (conjectually equivalent) characterizations of $\ks{\la}$:
Lapointe and Morse \cite{MR2331242} gave the Pieri-type formula,
and Lam \cite{Lam08} proved that $k$-Schur functions correspond to the Schubert basis of homology of the affine Grassmannian.
Moreover, Lam and Shimozono \cite{MR2923177} showed that
$k$-Schur functions play a central role in the explicit description of the Peterson isomorphism.

These developments have analogues in $K$-theory.
Lam, Schilling and Shimozono \cite{MR2660675} characterized the $K$-theoretic 
$k$-Schur functions as the Schubert basis of the $K$-homology of the affine
Grassmannian, and Morse \cite{Morse12} investigated them from a combinatorial
viewpoint, giving various properties including Pieri-type
formulas using affine set-valued strips (the form using cyclically
decreasing words are also given in \cite{MR2660675}).
In this paper 
we start from this combinatorial characterization (see Definition \ref{defi:KkPieri}).

Among the $k$-bounded partitions, 
those of the form $(t^{k+1-t})=(\underbrace{t,\dots,t}_{k+1-t})$ for $1\le t \le k$,
called {\it $k$-rectangle} and denoted by $\Rt$, play a special role.
A notable property is the \textit{$k$-rectangle factorization}
for $k$-Schur functions \cite[Theorem 40]{MR2331242}:
if a $k$-bounded partition has the form $R_t\cup \la$,
where the symbol $\cup$ denotes the operation of concatenating the two sequences and reordering the terms in the weakly decreasing order,
then the corresponding $k$-Schur function factorizes as follows:
\begin{equation}\label{intro:zero}
	\ks{R_t\cup\la} = \ks{R_t} \ks{\la}.
\end{equation}

It is natural to consider $K$-theoretic version of this formula.
For several reasons below,
in this regard it seems to make more sense to consider 
the sum of $K$-$k$-Schur functions $\sum_{\mu\le\la} \kks{\mu}$
rather than $K$-$k$-Schur function $\kks{\la}$
(here $\le$ denotes the \textit{strong order}, also known as the \textit{Bruhat order},
which is transferred from that of the affine symmetric group $\aSn$
through the bijection $\Pk\simeq\tSn$.
See Section \ref{sect:Prel::Coxeter::Order} and \ref{sect:Prel::AffSym::Bijection} for the detail):

\begin{itemize}
	\item
		\textit{Connection to $K$-Peterson isomorphism}.
	
		The (original) Peterson isomorphism,
		first presented by Peterson in his lectures at MIT
		and then published by Lam and Shimozono \cite{MR2600433},
		states that
		the homology of the affine Grassmannian %$H_*(Gr)$
		is isomorphic to
		the quantum cohomology of the flag variety %$QH^*(G/B)$
		after appropriate localization.
		As its $K$-theoretic version,
		an isomorphism between
		the $K$-homology of the affine Grassmannian and
		the quantum $K$-theory of the flag manifold, up to appropriate localization,
		is conjectured and called
		\textit{$K$-Peterson isomorphism}:
		
		\begin{itemize}
			\item
		In their attempt
		in \cite{1705.03435}
		to verify the coincidence of
		the Schubert structure constants
		in the $K$-homology of the affine Grassmannian and 
		the quantum $K$-theory of the flag manifold
		on torus-equivariant settings,
		Lam, Li, Mihalcea and Shimozono proved
		a special case of
		Theorem \ref{theo:StrongSumFactorization}
		for $SL_2$ (i.e.\,the case $k=1$) with explicit calculations,
		in the context of geometry:
		\begin{equation}\label{eq:factorization_O}
			\mathcal{O}_x \mathcal{O}_{t_{-\alpha^\vee}} = \mathcal{O}_{xt_{-\alpha^\vee}},
		\end{equation}
		where
		$x$ is any affine Grassmannian element in the affine Weyl group,
		$\mathcal{O}_x$ is the Schubert class of structure sheaves on the affine Grassmannian
		and $t_{-\alpha^\vee}$ is the translation by the negative of the simple coroot
		of $SL_2$.
		(See also Remark \ref{theo:rem_Rt}.)
		
			\item
		In \cite{doi:10.1093/imrn/rny051},
		Ikeda, Iwao and Maeno gave
		an explicit ring isomorphism,
		after appropriate localization,
		between the $K$-homology of the affine Grassmannian and
		the presentation
		of the quantum $K$-theory of the flag manifold
		that is conjectured by Kirillov and Maeno and
		proved by Anderson, Chen, and Tseng \cite{1711.08414},
		as well as a conjectural description
		of the image of the quantum Grothendieck polynomials,
		which is conjectured \Todo{ref} to be the quantum Schubert classes.
		These presentations notably involve
		the dual stable Grothendieck polynomials $g_{\Rt}$ 
		and their sum $\sum_{\mu\subset\Rt} g_{\mu}$
		corresponding to the $k$-rectangles $\Rt$.
		Note that $\mu\subset\Rt\iff\mu\le\Rt$, and that
		it is conjectured \Todo{ref} that $\kks{\la}=g_\la$ for $\la\subset\Rt$.
		\end{itemize}
	\item
		\textit{\Todo{remove natural}Natural appearances of $\sum_{\mu\le\la} \kks{\mu}$
			in $k$-rectangle factorization formulas of $\kks{\la}$}.
		
It is suggested in \cite[Remark 7.4]{MR2660675} that 
the $K$-$k$-Schur functions should also possess similar properties to (\ref{intro:zero}),
including the divisibility of $\kks{R_t\cup\la}$ by $\kks{R_t}$,
for which the author's preceding work \cite{Takigiku1,Takigiku2}
gives an affirmative answer.

Let us review the results of \cite{Takigiku1,Takigiku2}.
It is proved that
$\kks{R_t}$ divides $\kks{R_t\cup\la}$ 
in the ring $\Lk=\mathbb{Z}[h_1,\dots,h_k]$,
of which the $K$-$k$-Schur functions $\{\kks{\mu}\}_{\mu\in\Pk}$ form a basis.
However, unlike %the case of $k$-Schur functions 
(\ref{intro:zero}),
the quotient $\kks{R_t\cup\la}/\kks{R_t}$ is not a single term $\kks{\la}$
but in general a linear combination of $K$-$k$-Schur functions with leading term $\kks{\la}$:
for any $\la\in\Pk$,
\begin{equation}\label{eq:kks_divisibility}
	\kks{R_t \cup \la} 
	= \kks{R_t} \bigg(\kks{\la} + \sum_{\mu} a_{\la\mu}\kks{\mu}\bigg),
\end{equation}
summing over $\mu\in\Pk$ such that $|\mu|<|\la|$,
with some coefficients $a_{\la\mu}$ depending on $R_t$.
A special yet important case is the factorization of multiple $k$-rectangles:
for $1\le t\le k$ and $a>1$,
\[
	\kks{\Rt^a} = \kks{\Rt} \bigg(\sum_{\mu\subset\Rt}\kks{\mu}\bigg)^{a-1},
\]
where $\Rt^a=\Rt\cup\dots\cup\Rt$ ($a$ times).
Note that $\mu\subset\Rt \iff \mu\le\Rt$.
Furthermore, it is conjectured 
that
the set of $\mu$ appearing in (\ref{eq:kks_divisibility})
forms
an interval under the strong order: namely,
for any $\la \in \Pk$ and $1\le t \le k$, 
we expect there to exist $\nu \in \Pk$ such that
\begin{align*}
	&\kks{R_t\cup \la} = \kks{R_t} \sum_{\nu\le\mu\le\la} \kks{\mu}.
\end{align*}
\end{itemize}
These observations hint at merit of Definition \ref{defi:StrongSumKkSchur} below.

\subsection{Main results}\label{sect:Intro::Results}

Let $\le$, $\le_L$, and $\le_R$ be the strong, left weak, and right weak order on $\aSn$ (see Section \ref{sect:Prel::Coxeter::Order} for the detail).

From the observation above, we consider and denote by $\kkss{\la}$
the sum of $K$-$k$-Schur functions over the order ideal generated by $\la$ under the strong order $\le$:
\begin{defi} \label{defi:StrongSumKkSchur}
	For any $\la\in\Pk$, we write $\displaystyle\kkss{\la} = \sum_{\mu\le\la} \kks{\mu}$.
\end{defi}

Our first main theorem is a Pieri-type formula for $\kkss{\la}$.
We start with the Pieri rule for $\kks{\la}$ \cite{MR2660675,Morse12}:
for $\la \in \Pk$ and $1 \le r \le k$,
\[
		\kks{\la} h_r=
		\sum_{(A,\mu)} (-1)^{|\la|+r-|\mu|} \kks{\mu},
\]
summed over
affine set-valued strips $(\mu/\la,A)$ of size $r$
(See Definition \ref{defi:KkPieri} for more details).
In terms of $\kkss{\la}$,
this rule becomes relatively simple:
\begin{theo}\label{theo:StrongSumPieri}
	Let $\la \in \Pk$ and $1 \le r \le k$,
	and define $\wth_r = h_0 + h_1 + \dots + h_r$.
	Then
	\[
		\kkss{\la} \wth_{r} = \sum_{\mu} \kks{\mu},
	\]
	summed over $\mu\in\Pk$ such that $\mu \le \ka$ for some $\ka\in\Pk$
	such that $\ka/\la$ is a weak strip of size $r$.
\end{theo}
To express its right-hand side as a linear combination of $\{\kkss{\mu}\}_{\mu}$,
we recall that
a weak strip over $\la$ corresponds to a proper subset of $I=\{0,1,\dots,k\}$:
for $\ka\in\Pk$, 
$\ka/\la$ is a weak strip
if and only if
there exists $A\subsetneq I$ such that $\ka = d_A \la \ge_L \la$,
where 
$d_A$ is the cyclically decreasing permutation corresponding to $A$
(see Section
\ref{sect:Prel::AffSym::CycDecrElem},
\ref{sect:Prel::AffSym::Bijection}, and
\ref{sect:Prel::AffSym::WeakStrip}
for the detail).

\begin{theo}\label{theo:StrongSumPieri'}
	With the setting in Theorem \ref{theo:StrongSumPieri},
	we let $d_{A_1}\la, d_{A_2}\la, \dots$ be the list of weak strips of size $r$ over $\la$.
	Then
	\begin{align*}
		\kkss{\la} \wth_{r} 
		&= \sum_{m\ge 1} (-1)^{m-1}
			\sum_{a_1 < \dots < a_m}
				\kkss{d_{A_{a_1}\cap\dots\cap A_{a_m}}\la}. \\
		\bigg(&= \sum_{a} \kkss{d_{A_a}\la}
			- \sum_{a<b} \kkss{d_{A_a\cap A_b}\la}
			+ \sum_{a<b<c} \kkss{d_{A_a\cap A_b\cap A_c}\la}
			- \cdots\bigg)
	\end{align*}
	(Moreover $d_{A_a\cap A_b\cap \dots}\la = 
		(d_{A_a}\la)\wedge(d_{A_b}\la)\wedge\dots$,
		where $\wedge$ denotes the meet in the poset $\Pk$ with the strong order.
	See also Proposition \ref{theo:alias:Zu_props}.)
\end{theo}

Our second main theorem is the $k$-rectangle factorization formula for $\kkss{\la}$,
which holds in the same form as that for $k$-Schur functions (\ref{intro:zero}):
\begin{theo}\label{theo:StrongSumFactorization}
	For any $\la\in\Pk$ and $1 \le t \le k$, we have
	$$
		\kkss{\Rt\cup\la} = \kkss{\Rt} \kkss{\la}.
	$$
\end{theo}

To deduce Theorem \ref{theo:StrongSumFactorization}
from Theorem \ref{theo:StrongSumPieri'}
is easy and discussed in Section \ref{sect:StrongSumFactorization}.
The proof of Theorem \ref{theo:StrongSumPieri} and \ref{theo:StrongSumPieri'},
on the other hand, is the technical heart of this paper and
requires auxiliary work on the strong and weak orderings on the set of affine permutations as well as the structure of the set of weak strips,
which are discussed in Section \ref{sect:PropOrder} and \ref{sect:PropWS}.

\vspace{2mm}

This paper is organized as follows.

In Section \ref{sect:Prel},
we review notations and facts on
combinatorial backgrounds.
In Section \ref{sect:Prel::Coxeter}
we treat arbitrary Coxeter groups and its strong and weak orderings.
It also contains quick reviews on the generalized quotients \cite{MR946427}
and the Demazure products.
Section \ref{sect:Prel::AffSym}
contains notations specific to the affine symmetric groups and
a review on its Young-diagrammatic treatment.
In Section \ref{sect:Prel::SymFcn} we briefly review the Pieri-type formulas for
$k$-Schur and $K$-$k$-Schur functions.

Section \ref{sect:PropOrder} contains
technical lemmas
on the strong and weak orders on arbitrary Coxeter groups.
In Section \ref{sect:PropOrder::Lattice}
the lattice property of the weak order is reviewed.
Although 
it is known that the quotient of an affine Weyl group by its corresponding finite Weyl group forms a lattice under the weak order \cite{MR1740744},
we include another proof for the type affine A using the $k$-Schur functions.
Section \ref{sect:PropOrder::phipsi}
contains basic properties of the Demazure and anti-Demazure actions.
In Section \ref{sect:PropOrder::SLjoin}
we show the existence of
$\min_\le\{z\in W\mid x\le z\ge_L y\}$ and
$\max_\le\{z\in W\mid x\ge_L z\le y\}$,
analogous to the join and meet.
In Section \ref{sect:PropOrder::flip}
we consider an ``interval-flipping'' map
$\Phi_z\colon[e,z]_L\lra[e,z]_R;\,x\mapsto zx^{-1}$
and show that
$\Phi_z$ is anti-isomorphic under the strong order and
sends strong-meets (if exist) to strong-joins.
In Section \ref{sect:PropOrder::CP}
we show the Chain Property of lower weak intervals,
analogous to the Chain Property of the generalized quotients.

In Section \ref{sect:PropWS},
we focus on the affine symmetric groups and
give results on the structure of the set of weak strips,
which includes:
\begin{prop}[$\subset$ Proposition \ref{theo:Zu_props}]\label{theo:alias:Zu_props}
	For any $\la\in\Pk$ and
	$A,B\subsetneq I$ with $d_A\la/\la$ and $d_B\la/\la$ are weak strips,
	\begin{enumerate}[label=\textup{(\arabic*)}]
		\item
			$d_{A\cap B}\la/\la$ and $d_{A\cup B}\la/\la$ are weak strips.
		\item
			$d_{A\cap B}\la = d_A\la \wedge d_B\la$ under the strong order.
	\end{enumerate}
\end{prop}
 
\begin{prop}[$\subset$ Proposition \ref{theo:forbiddenindex}]
	For any $\la\in\Pk$,
	there exists $i_\la\in I$ $(=\{0,1,\dots,k\})$
	such that
	$i_\la \notin A$
	for any weak strip $d_A\la/\la$.
\end{prop}

Section \ref{sect:StrongSumPieri} and \ref{sect:StrongSumFactorization}
are devoted to proving 
the Pieri-type formula for $\kkss{\la}$
(Theorem \ref{theo:StrongSumPieri} and \ref{theo:StrongSumPieri'})
and the $k$-rectangle factorization formula for $\kkss{\la}$
(Theorem \ref{theo:StrongSumFactorization}), respectively.

\subsection*{Acknowledgement}
The author would like to express his gratitude to
T.\ Ikeda
for suggesting this topic to the author,
many fruitful discussions
and communicating to him the idea of considering the Schubert class of structure sheaves,
related to the work \cite{1705.03435}.
He is grateful to
I.\ Terada
for many valuable discussions and comments.
He also wishes to thank
Hiroshi Naruse and
Mark Shimozono
for helpful comments.
This work was supported by
the Program for Leading Graduate
Schools, MEXT, Japan.

\section{Preliminaries}\label{sect:Prel}

In this section we review some requisite combinatorial backgrounds.

\subsection{Coxeter groups}\label{sect:Prel::Coxeter}

For basic definitions for the Coxeter groups we refer the reader to \cite{MR2133266} or \cite{MR1066460}.

\subsubsection{Strong and weak orderings}\label{sect:Prel::Coxeter::Order}
Let $(W,S)$ be a Coxeter group and $T=\{wsw^{-1}\mid w\in W\}$ its set of reflections.
The \textit{left weak order} (or simply \textit{left order}) $\le_L$,
\textit{right weak order} (or \textit{right order}) $\le_R$, and
\textit{strong order} (or \textit{Bruhat order}) $\le$
 on $W$ are generated by
the covering relations:
\begin{align*}
	u\lecovL v &\iff l(v)=l(u)+1,\ v=s u \text{ for some } s\in S, \\
	u\lecovR v &\iff l(v)=l(u)+1,\ v=u s \text{ for some } s\in S, \\
	u\lecov v &\iff l(v)=l(u)+1,\ v=tu \text{ for some } t\in T.
\end{align*}
Note that the definition of the strong order looks different from but coincides with the classical one.

%%% basic property of the weak order
It is a few immediate observations that, for $u,v\in W$,
\begin{gather}
	u\le_L v \iff l(vu^{-1})+l(u) = l(v), \\
	u\le_R v \iff l(u)+l(u^{-1}v) = l(v), \\
	u\le_R uv \iff l(u)+l(v)=l(uv) \iff v\le_L uv.
\end{gather}
We often use these equivalences without any mention.
Using this translation from the weak order to length conditions,
we can easily prove the following lemma:
\begin{lemm}\label{theo:weakorder_xyz}
	For $x,y,z\in W$, we have
	\begin{enumerate}[label=\textup{(\arabic*)}]
		\item
			$z\le_L yz \le_L xyz$
			$\iff$
			$y\le_L xy$ and $z\le_L xyz$.
		\item
			$z\ge_L yz \ge_L xyz$
			$\iff$
			$y\le_L xy$ and $z\ge_L xyz$.
	\end{enumerate}
\end{lemm}

We often use the following notation taken from \cite{MR946427}:
for $w\in W$ we let $\redex{w}$ denote any reduced expression for $w$,
and $\redprod{u}{v}$ the concatenation of reduced expressions for $u$ and $v$.
Hence, saying that $\redex{u}\redex{v}$ is reduced means $l(u) + l(v) = l(uv)$.

%%% intervals
%
For $u,v\in W$ with $u\le_L v$ the set $\{w\in W\mid u\le_L w\le_L v\}$ is called a \textit{left interval} and denoted by $[u,v]_L$.
We define \textit{right interval} $[u,v]_R$ and \textit{strong (or Bruhat) interval} $[u,v]$ similarly.
We shall use the notation $[u,\infty)_L$ to denote the set $\{w\in W\mid u\le_L w\}$,
and define $[u,\infty)_R$ and $[u,\infty)$ similarly.

%%% well-known properties
%
In this paper we heavily use some well-known results on the strong and weak orderings on Coxeter groups described below.
See, for example, \cite{MR2133266} for details.
Let $v,w\in W$.

\vspace{2mm}
\noindent\textit{Strong Exchange Property.}
Suppose $w=s_1 s_2 \dots s_k$ ($s_i\in S$) and $t\in T$.
If $l(tw)<l(w)$, then $tw=s_1\dots \widehat{s_i}\dots s_k$ for some $i\in[k]$.
Furthermore, if $s_1 s_2 \dots s_k$ is a reduced expression then $i$ is uniquely determined.

\vspace{2mm}
\noindent\textit{Subword Property.}
Let $w=s_1 s_2\dots s_k$ be a reduced expression.
Then $v\le w$ if and only if there exists a reduced expression 
$v = s_{i_1} s_{i_2} \dots s_{i_l}$ with $1\le i_1<i_2<\dots<i_l\le k$.

\vspace{2mm}
\noindent\textit{Chain Property.}\footnote{With the definition of $\le$ we employed here, this says nothing.}
If $v\le w$, then there exists a chain
$v=x_0\lecov x_1\lecov\dots\lecov x_k=w$.

\vspace{2mm}
\noindent\textit{Lifting Property} (also known as \textit{Z-property}).
Let $s\in S$. If $sv>v$ and $sw>w$, then the following are equivalent:
(1) $v\le w$,
(2) $v\le sw$, and
(3) $sv\le sw$.

\subsubsection{Generalized quotients}\label{sect:Prel::Coxeter::GenQuot}

For $V\subset W$, let
$W/V=\{w\in W\mid l(wv) = l(w) + l(v) \text{ for all } v\in V\}$.
The subsets of the form $W/V$ are called \textit{(right) generalized quotients} \cite{MR946427}.
Similarly the set of the form
$V\backslash W=\{w\in W\mid l(vw) = l(v) + l(w) \text{ for all } v\in V\}$
is called \textit{left generalized quotients}.
Note that, when $V = W_J$, the parabolic subgroup corresponding to $J\subset I$,
the generalized quotient $W/W_J$ is just the parabolic quotient $W^J$.

It is shown in \cite[Lemma 2.2]{MR946427} that
if $a,b,v\in W$ satisfy $l(av) = l(a) + l(v)$ and $l(bv)=l(b)+l(v)$,
then $av < bv \iff a < b$.
An immediate consequence is
\begin{equation}\label{eq:gq_isom}
	W/\{v\}\simeq[v,\infty)_L; w\mapsto wv
\end{equation}
under both the strong and left weak order.

%%% Chain Property for generalized quotients
%
\vspace{2mm}
\noindent\textit{Chain Property for generalized quotients $($\cite[Corollary 3.5]{MR946427}$)$.}
If $v,w\in W/V$ and $v<w$, then there exists a chain
$v=x_0\lecov x_1\lecov\dots\lecov x_k=w$ with $x_i\in W/V$ for all $i$.

%%% 0-Hecke algebra
%
\subsubsection{$0$-Hecke algebra and Demazure product}\label{sect:Prel::Coxeter::0Hecke}

The \textit{$0$-Hecke algebra} $H$ associated to $(W,S)$ is the associative algebra
generated by $\{v_s \mid s\in S\}$ subject to the quadratic relations
$v_s^2= -v_s$
and the braid relations of $(W,S)$, that is,
$\underbrace{v_s v_t v_s \dots}_{m} = \underbrace{v_t v_s v_t \dots}_{m}$
for $s,t\in S$ with 
$\underbrace{s t s \dots}_{m} = \underbrace{t s t \dots}_{m}$ .
For $w\in W$ we can define without ambiguity
$v_w\in H$ to be $v_{s_1}\dots v_{s_n}$ where $s_{1}\dots s_{n}$ is any reduced expression for $w$. 
Furthermore, the elements $\{v_w \mid w\in W\}$ form a basis of $H$.
The \textit{Demazure product} (or \textit{Hecke product}) $*$ on $W$ describes the multiplication of basis elements in $H$:
$x*y$ is such that $v_x v_y =\pm v_{x*y}$.
Some properties on the Demazure product can be found on
\cite{KNUTSON2004161,buch2015}.

We explicitly prepare the notation to denote the left multiplication in the Demazure product:
for $s\in S$, we define the Demazure action $\phi_s \colon W\lra W$ by
\[
	\phi_s(x) = s*x =
		\begin{cases}
			x & \text{(if $x > s x$)} \\
			s x & \text{(if $x < s x$)}
		\end{cases}.
\]
Similarly we define the anti-Demazure action $\psi_s \colon W\lra W$ by
\[
	\psi_s(x) =
		\begin{cases}
			s x & \text{(if $x > s x$)} \\
			x & \text{(if $x < s x$)}
		\end{cases}.
\]
These maps $\{\phi_s\}_s$ and $\{\psi_s\}_s$ satisfy
the quadratic relations $\phi_s^2=\phi_s$, $\psi_s^2=\psi_s$ and
the braid relations of $(W,S)$;
a direct proof (found on \cite[Proposition 2.1]{MR2310418}) of this (for $\psi$) is that
both $\psi_s \psi_t \psi_s \dots$ and $\psi_t \psi_s \psi_t \dots$ ($m$ terms for each),
where $sts\dots=tst\dots$ ($m$ terms for each), 
send $x\in W$ to the shortest
(resp.\,longest, when we consider $\phi$)
element of the parabolic coset $W_{\{s,t\}}x$.
Therefore we can define without ambiguity $\phi_x, \psi_x:W\lra W$ for $x\in W$ by
$\phi_x=\phi_{s_1}\dots\phi_{s_n}$ and 
$\psi_x=\psi_{s_1}\dots\psi_{s_n}$ where
$x=s_{1}\dots s_{n}$ is any reduced expression.
Similarly we define right Demazure and anti-Demazure actions
$\phi_s^R,\psi_s^R\colon W\lra W$ for $s\in S$ by
$\phi_s^R(x) = \phi_s(x^{-1})^{-1}$ and
$\psi_s^R(x) = \psi_s(x^{-1})^{-1}$,
that is, $\phi_s^R(x)=xs$ if $x<xs$, etc.
We also define $\phi_x^R$ and $\psi_x^R$ to be
$\phi_{s_n}^R \dots \phi_{s_1}^R$ and
$\psi_{s_n}^R \dots \psi_{s_1}^R$
(be careful for the order of composition)
where 
$x=s_{1}\dots s_{n}$ is any reduced expression.
Note that $\phi_x(y)=x*y=\phi_y^R(x)$.
When $S$ is indexed with a set $I$, i.e.\,$S=\{s_i\mid i\in I\}$,
we often write $\phi_i=\phi_{s_i}$ and $\psi_i=\psi_{s_i}$.

The following lemma is essentially given in \cite[Theorem 4.2]{MR946427},
and explicitly in \cite[Proposition 3.1(e)]{buch2015}:
\begin{lemm}\label{theo:b_xyz}
	Let $x,y,z\in W$ with $x*y = z$, that is, $\phi_x(y)=z=\phi_y^R(x)$.
	Let $x'=zy^{-1}$ and $y'=x^{-1}z$, that is, $z=xy'=x'y$.
	Then we have
	\begin{enumerate}[label=\textup{(\arabic*)}]
		\item
			$x, x'\le_{R} z$
		\item
			$y, y'\le_{L} z$,
		\item
			$l(z) = l(x)+l(y') = l(x')+l(y)$.
		\item
			$x'\le x$.
		\item
			$y'\le y$.
	\end{enumerate}
\end{lemm}
\begin{proof}
	It follows easily from the definition of $*$ and the Subword Property.
\end{proof}
The proof of the following lemma is easy and similar to that of Lemma \ref{theo:b_xyz}.
\Todo{really?}
\begin{lemm}\label{theo:psi_xyz}
	Let $x,y,z\in W$ with $\psi_x(y)=z$.
	Let $x'=zy^{-1}$, that is, $z=x'y$.
	Then we have
	\begin{enumerate}[label=\textup{(\arabic*)}]
		\item
			$x'\le x$.
		\item
			$z\le_L y$.
		\item
			$x'^{-1}\le_R y$.
	\end{enumerate}
\end{lemm}

We see more properties on $\phi_x,\psi_x$ in Section \ref{sect:PropOrder::phipsi}.

%%% Affine Symmetric Groups
%
\subsection{Affine symmetric groups}\label{sect:Prel::AffSym}

In this section we briefly review the connection between
affine permutations, bounded partitions and core partitions. 
We refer the reader to
\cite[Chapter 2]{MR3379711} and \cite{MR3001656}
for the details.

\vspace{2mm}
\textit{Hereafter we fix a positive integer $k$.}

\subsubsection{Affine symmetric group}\label{sect:Prel::AffSym::AffSym}

Let $I=\mathbb{Z}_{k+1}=\{0,\dots,k\}$.
Let $[p,q]=\{p,p+1,\dots,q-1,q\}\subsetneq I$ for $p\neq q-1$.
For example, $[4,2]=\{4,5,0,1,2\}$ where $k=5$.
A subset $A\subset I$ is called \textit{connected} if $A=[p,q]$ for some $p,q$.
A \textit{connected component} of $A\subsetneq I$ is a maximal connected subset of $A$.

%%% \ti Sn
%
The {\it affine symmetric group} $\ti S_{k+1}$ is a group generated by 
the generators $\{ s_i\mid i\in I\}$
subject to the relations
$s_i^2=1$,
$s_i s_{i+1} s_i = s_{i+1} s_i s_{i+1}$,
$s_i s_j = s_j s_i$ for $i-j \not\equiv 0,\pm 1$,
with all indices considered mod $(k+1)$.
We often write $s_{ij\dots}$ instead of $s_i s_j \cdots$.
The parabolic quotient $\aSn/S_{k+1}$,
where $S_{k+1}$ is the symmetric group $\langle s_1,\dots,s_k\rangle$ as a subgroup of $\aSn$,
is denoted by $\aG$
and its elements are called {\it affine Grassmannian elements}.

%%% Descent sets
%
For $x\in \aSn$, 
the set of \textit{right descents} $D_R(x)$ is $\{i\in I\mid x>xs_i\}$ ($\subsetneq I$).
The set of \textit{left descents} $D_L(x)$ is defined similarly.
%
%%% i-dominant
%
For $i\in I$, an element $w\in\aSn$ is called \textit{$i$-dominant} if $D_R(w)\subset\{i\}$.
Note that an affine permutation is $0$-dominant if and only if it is affine Grassmannian.

%%% cyclically decreasing elements
%
\subsubsection{Cyclically decreasing elements}\label{sect:Prel::AffSym::CycDecrElem}
A word $a=a_1 a_2\dots a_m$ with letters from $I$ is called 
\textit{cyclically decreasing} (resp.\,\textit{cyclically increasing})
if $a_1,a_2,\dots,a_m$ are distinct and
any $j\in I$ does not precede $j+1$ (resp.\,$j-1$) in $a$.
For $A\subsetneq I$, the \textit{cyclically decreasing element} $d_A$ is defined to be
$s_{i_1}s_{i_2}\dots s_{i_m}$ where $A=\{i_1,i_2,\dots,i_m\}$ and
the word $i_1 i_2\dots i_m$ is cyclically decreasing.
The \textit{cyclically increasing element} $u_A=s_{i_m}s_{i_{m-1}}\dots s_{i_1}$ is defined similarly.
Note that these definitions are independent of the choice of the word.

\begin{exam}
Let $k=5$ and $A=\{0,1,3,5\}\subsetneq\Z_6$.
The possible cyclically decreasing words for $A$ are
$1053$, $1035$, $1305$ and $3105$, and hence
$d_A = s_1 s_0 s_5 s_3 = s_1 s_0 s_3 s_5 = s_1 s_3 s_0 s_5 = s_3 s_1 s_0 s_5$.
\end{exam}

%%% Pk ~ Cn ~ tSn0
%
\subsubsection{Connection to bounded partitions and core partitions}\label{sect:Prel::AffSym::Bijection}

In this section we review the bijection between
the set of $k$-bounded partitions,
$k+1$-core partitions and
affine Grassmannian elements in $\aSn$.
For the details see \cite[Chapter 2]{MR3379711} and references given there.

%%% Pk
A partition $\lambda$ is called {\it $k$-bounded} if $\lambda_1 \le k$.
Let $\Pk$ be the set of all $k$-bounded partitions.
%
%%% Cr
An {\it $r$-core} (or simply a \textit{core} if no confusion can arise) is a partition none of whose cells have a hook length equal to $r$.
We denote by $\mathcal{C}_r$ the set of all $r$-core partitions.

%%% The bijection
%
Now we recall the
bijection
\begin{equation}\label{eq:PCS}
	\Pk\simeq\Cn\simeq\aG.
\end{equation}

The map
	$\bdd \colon \Cn \lra \Pk ; \ka \mapsto \la$
is defined by
$\la_i=\#\{j\mid (i,j)\in\ka,\ \hook{(i,j)}{\ka}\le k\}$.
In fact $\bdd$ is bijective and
the inverse map $\core = \bdd^{-1} \colon \Pk \lra \Cn$
is algorithmically described as a ``sliding cells'' procedure.

%%% Sn -> Cn
The map $\StoC \colon \aG \lra \Cn$ is constructed via an action of $\aSn$ on $\Cn$:
for $\ka\in\Cn$ and $i\in I$,
we define $s_i\cdot \ka$ to be
$\ka$ with all its addable (resp.\,removable) corners with residue $i$ added (resp.\,removed),
where the \textit{residue} of a cell $(i,j)$ is $j-i\mod k+1$.
In fact 
this gives a well-defined $\aSn$-action on $\Cn$,
which induces the bijection
$\StoC : \aG \longrightarrow \Cn ; w \mapsto w\cdot \emptyset$.

The map $\Pk\lra\aG; \la\mapsto w_\la$ is given by
$w_\la = s_{i_1} s_{i_2}\dots s_{i_l}$,
where $(i_1,i_2,\dots,i_l)$ is the sequence
obtained by reading the residues of the cells in $\la$, from the shortest row to the largest, and within each row from right to left.
See \cite[Corollary 48]{MR2167475} for the proof.

%%% k-transpose
%
For $\la\in\Pk$, the \textit{$k$-transpose} of $\la$ is $\bdd(\core(\la)')$ and denoted by $\la^{\omega_k}$.

\begin{exam}
Let $k=3$ and $\la=(3,2,1)\in\Pk[3]$.
The corresponding
$4$-core partition
and affine permutation are
$\core(\la)=(5,2,1)\in\Cn[4]$ and $w_\la=s_{203210}\in\aG[4]$.
(See Figure \ref{fig:core_bdd_perm}.)
\end{exam}
\begin{figure}
\centering
\begin{tikzpicture}[]
	\node (core) at (0,0) {
		\begin{tikzpicture}[scale=0.4]
			\tikzyfill{2}
			\tikzydiag{5,2,1}
			\tikzydnum{\small}{1}{0,1,2,3,0}
			\tikzydnum{\small}{2}{3,0}
			\tikzydnum{\small}{3}{2}
		\end{tikzpicture}
	};
	\node (bdd) at (3,0) {
		\begin{tikzpicture}[scale=0.4]
			\tikzydiag{3,2,1}
			\tikzydnum{\small}{1}{0,1,2}
			\tikzydnum{\small}{2}{3,0}
			\tikzydnum{\small}{3}{2}
		\end{tikzpicture}
	};
	\node (perm) at (6,0) {
		$s_{203210}$
	};
	\node [below of=core] {$\core(\la)$};
	\node [below of=bdd] {$\la$};
	\node [below of=perm] {$w_\la$};
	\draw [<->] (core) -- (bdd);
	\draw [<->]	(bdd) -- (perm);
\end{tikzpicture}
\caption{$k=3$, $\la=(3,2,1)\in\Pk[3]$,
$\core(\la)=(5,2,1)\in\Cn[4]$,
and $w_\la=s_{203210}\in\aG[4]$.}
\label{fig:core_bdd_perm}
\end{figure}

\subsubsection{Weak strips} \label{sect:Prel::AffSym::WeakStrip}

\begin{defi}\label{defi:weakstrip}
	For $v,w\in\aG$, we call $v/w$ is a \textit{weak strip} (or \textit{affine strip}) of size $r$ if
	$v=d_Aw\ge_L w$ for some $A\subsetneq I$ with $|A|=r$.
	We also say $v$ is a weak strip of size $r$ over $w$.
\end{defi}

\begin{defi}\label{defi:asvstrip}
	For $v,w\in\aG$ and $A\subsetneq I$,
	we call $(v/w,A)$ is an \textit{affine set-valued strip} of size $r$ if
	$v=d_A*w\,(=\phi_{d_A}(w))$ and $|A|=r$.
	We also say $(v,A)$ is a affine set-valued strip of size $r$ over $w$.
\end{defi}

Note that if $(v/w,A)$ is an affine set-valued strip of size $r$ then
$v/w$ is an affine strip of size $\le r$.

\begin{rema}
Idetifying $\la$, $\core(\la)$ and $w_\la$ through
the bijection $\Pk\simeq\Cn\simeq\aG$,
we often say $\mu/\la$ (resp.\,$\ka/\gamma$) is a weak strip for
$\la,\mu\in\Pk$ (resp.\,$\ka,\gamma\in\Cn$), etc.
\end{rema}

\begin{rema}
Regarding $v,w\in\aG$ as bounded (or core) partitions as above,
we see these notions are variants of the horizontal strip.
For example, 
$w_\mu/w_\la$ is a weak strip if and only if
the corresponding cores $\core(\mu)/\core(\la)$ form a horizontal strip and
$w_\mu\ge_L w_\la$,
and the term ``affine set-valued'' originates in affine set-valued tableaux.
See, for example, \cite{MR3379711,Morse12} for more details.
\Todo{more details}
\end{rema}

\begin{exam}
\label{exam:ws}
Let $k=3$ and $\la=(3,2,1)\in\Pk[3]$,
and thus $w_\la = s_{203210}$ and $\core(\la)=(5,2,1)$.
Figure \ref{fig:weakstrips} lists 
all $v$ such that $v/w_\la$ is a weak strip (the corresponding core partitions are displayed).
\end{exam}

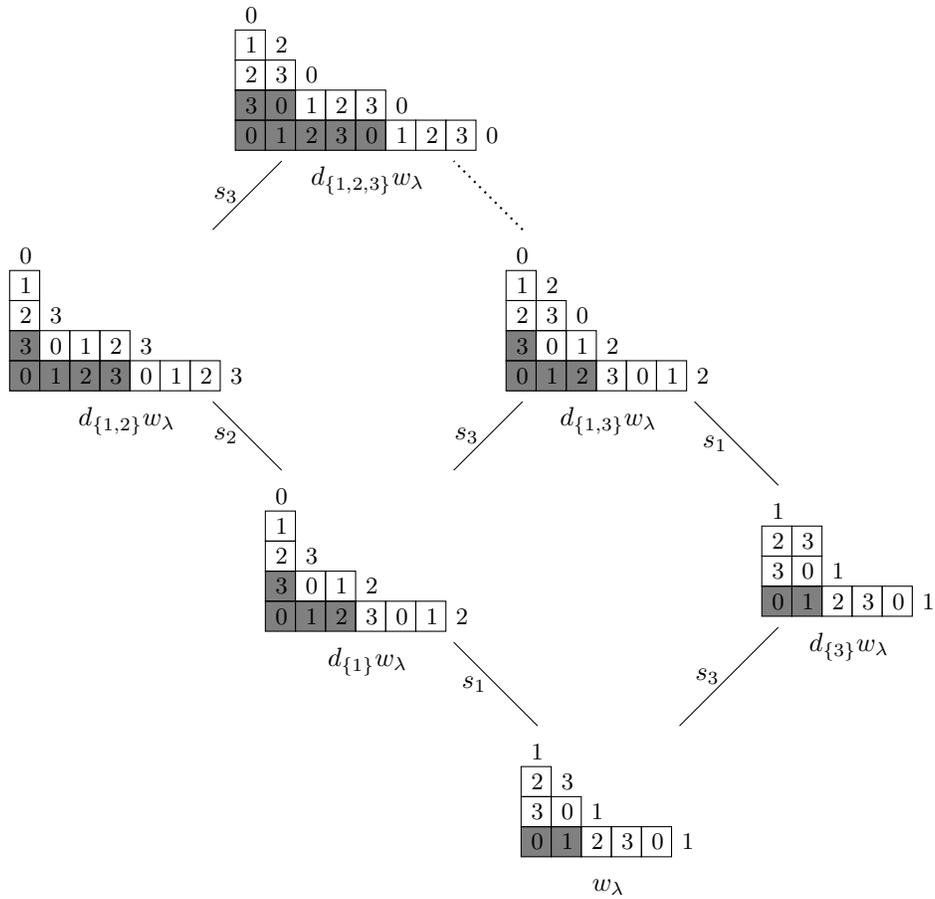
\begin{figure}
	\centering
	\begin{tikzpicture}
    	\node (Z) at (0,0) {
        	\begin{tikzpicture}[scale=1.6]
        		\node (la) at (0,0) {
					\tikzitem[0.4]{
						\tikzyfill{2};
						\tikzydiag{5,2,1};
						\tikzvnum{1}{0,1,2,3,0,1};
						\tikzvnum{2}{3,0,1};
						\tikzvnum{3}{2,3};
						\tikzvnum{4}{1};
					}
				};
				\node [below of=la,yshift=-2mm] {$w_\la$};
        		\node (1) at (-2,2) {
					\tikzitem[0.4]{
						\tikzyfill{3,1};
						\tikzydiag{6,3,1,1};
						\tikzvnum{1}{0,1,2,3,0,1,2};
						\tikzvnum{2}{3,0,1,2};
						\tikzvnum{3}{2,3};
						\tikzvnum{4}{1};
						\tikzvnum{5}{0};
					}
				};
				\node [below of=1,yshift=-4mm] {$d_{\{1\}}w_\la$};
        		\node (3) at (2,2)  {
					\tikzitem[0.4]{
						\tikzyfill{2};
						\tikzydiag{5,2,2};
						\tikzvnum{1}{0,1,2,3,0,1};
						\tikzvnum{2}{3,0,1};
						\tikzvnum{3}{2,3};
						\tikzvnum{4}{1};
					}
				};
				\node [below of=3,yshift=-2mm] {$d_{\{3\}}w_\la$};
        		\node (31) at (0,4) {
					\tikzitem[0.4]{
						\tikzyfill{3,1};
						\tikzydiag{6,3,2,1};
						\tikzvnum{1}{0,1,2,3,0,1,2};
						\tikzvnum{2}{3,0,1,2};
						\tikzvnum{3}{2,3,0};
						\tikzvnum{4}{1,2};
						\tikzvnum{5}{0};
					}
				};
				\node [below of=31,yshift=-4mm] {$d_{\{1,3\}}w_\la$};
        		\node (21) at (-4,4)  {
					\tikzitem[0.4]{
						\tikzyfill{4,1};
						\tikzydiag{7,4,1,1};
						\tikzvnum{1}{0,1,2,3,0,1,2,3};
						\tikzvnum{2}{3,0,1,2,3};
						\tikzvnum{3}{2,3};
						\tikzvnum{4}{1};
						\tikzvnum{5}{0};
					}
				};
				\node [below of=21,yshift=-4mm] {$d_{\{1,2\}}w_\la$};
        		\node (321) at (-2,6) {
					\tikzitem[0.4]{
						\tikzyfill{5,2};
						\tikzydiag{8,5,2,1};
						\tikzvnum{1}{0,1,2,3,0,1,2,3,0};
						\tikzvnum{2}{3,0,1,2,3,0};
						\tikzvnum{3}{2,3,0};
						\tikzvnum{4}{1,2};
						\tikzvnum{5}{0};
					}
				};
				\node [below of=321,yshift=-4mm] {$d_{\{1,2,3\}}w_\la$};
        		
        		\draw (la) -- node[left] {$s_1$} (1);
        		\draw (la) -- node[left] {$s_3$} (3);
        		\draw (1) -- node[left] {$s_3$} (31);
        		\draw (3) -- node[left] {$s_1$} (31);
        		\draw (1) -- node[left] {$s_2$} (21);
        		\draw (21) -- node[left] {$s_3$} (321);
        		\draw [thick, dotted] (31) -- (321);
        	\end{tikzpicture}
    	};
%		\node [below of=Z, yshift=-50mm] {$Z_{u,+}$};
	\end{tikzpicture}
	\caption{
		The weak strips over $w_{\la}$ where $\la=(3,2,1)$.
		Left weak covers are represented as solid lines,
		and strong covers are solid or dotted lines.
		A solid edge between $v$ and $w$ is labelled with $s_i$ if
		$v=s_i w$.
	}
	\label{fig:weakstrips}
\end{figure}

%%% k-code
%
\subsubsection{$k$-code}\label{sect:Prel::AffSym::k-code}

The content of this section is mostly cited from \cite{MR3001656}.

A \textit{$k$-code} is a function $\alpha:I\lra \Z_{\ge 0}$ such that
there exists at least one $i\in I$ with $\alpha(i)=0$.
We often write $\alpha_i=\alpha(i)$.
%
%%% diagram
%
The \textit{diagram} of a $k$-code $\alpha$ is the Ferrers diagram on a cylinder with $k+1$ columns indexed by $I$,
where the $i$-th column contains $\alpha_i$ boxes.
%
%%% filling
%
A $k$-code $\alpha$ may be identified with its \textit{filling},
which is the diagram of $\alpha$ with all its boxes marked with
their residues, that is, $i-j$ ($\in I$) for one in the $i$-th column and $j$-th row.
A \textit{flattening} of the diagram of a $k$-code $\alpha$ is what is obtained by cutting out a column with no boxes (that is, column $j$ with $\alpha_j=0$).
A \textit{reading word} of $\alpha$ is obtained by reading the rows of the flattening of $\alpha$ from right to left, beginning with the last row.
Note that, though a $k$-code may have multiple columns with no boxes,
the affine permutation given by the reading word of $\alpha$ is independent of the choice of a flattening.
Indeed, for a $k$-code $\alpha$ with $m$ rows,
letting $A_i$ be the set of the residues of the boxes in the $i$-th row in the diagram of $\alpha$,
we have that $d_{A_m}\dots d_{A_2} d_{A_1}$ is the affine permutation corresponding to $\alpha$.
In fact this correspondence is bijective (Theorem \ref{theo:RD_bij});
an algorithm to obtain a $k$-code from an affine permutation is explained below.

\vspace{2mm}
\noindent\textit{Maximizing moves.}

For a cyclically decreasing decomposition
$w=d_{A_m}\dots d_{A_1}$,
there corresponds a ``skew $k$-code diagram'', that is,
a set of boxes in the cylinder with $k+1$ columns indexed by $I$
for which $A_i$ is the set of the residues of the boxes in the $i$-th row.
To justify it to the bottom, we consider the following ``two-row move'':
pick any consecutive two rows $A_{a}$ and $A_{a+1}$, and let $i,j\in I$ with $j\neq i-1$.
Then,
\begin{enumerate}[label=\textup{(\arabic*)}]
	\item
		if $i-1\notin A_{a+1}$, $[i,j]\subset A_{a+1}$, 
		$[i+1,j]\subset A_{a}$, and $i,j+1\notin A_{a}$,
		then we replace $A_{a}$ and $A_{a+1}$ with
		$A_{a}\cup\{i\}$ and $A_{a+1}\sm\{j\}$,
		reflecting the equation
		$(s_{j}s_{j-1}\dots s_{i})(s_{j}\dots s_{i+1}) = 
		(s_{j-1}\dots s_{i})(s_{j}\dots s_{i+1}s_{i})$.
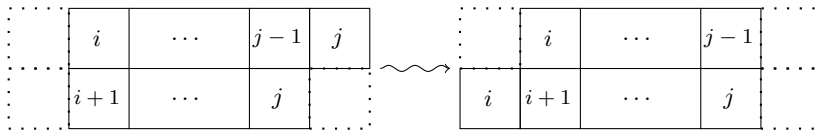
		\[
			\tikzitem{
            	\node (1a) at (0,0){
            		\begin{tikzpicture}[scale=0.8]
            			\tikzbox[thick, loosely dotted]{1}{1}
            			\tikzbox[thick, loosely dotted]{2}{1}
            			\tikzbox[thick, loosely dotted]{1}{6}
            			\draw (2,0) -- (4,0);
            			\draw (2,1) -- (4,1);
            			\draw (2,2) -- (4,2);
            			\tikznum{1}{3.5}{$\dots$}
            			\tikznum{2}{3.5}{$\dots$}
            			
            			\tikzboxnum[\footnotesize]{1}{2}{$i+1$}
            			\tikzboxnum{2}{2}{$i$}
            			\tikzboxnum{1}{5}{$j$}
            			\tikzboxnum[\footnotesize]{2}{5}{$j-1$}
            			\tikzboxnum{2}{6}{$j$}
            		\end{tikzpicture}
            	};
            	\node (1b) at (6,0){
            		\begin{tikzpicture}[scale=0.8]
            			\tikzbox[thick, loosely dotted]{2}{1}
            			\tikzbox[thick, loosely dotted]{1}{6}
            			\tikzbox[thick, loosely dotted]{2}{6}
            			\draw (2,0) -- (4,0);
            			\draw (2,1) -- (4,1);
            			\draw (2,2) -- (4,2);
            			\tikznum{1}{3.5}{$\dots$}
            			\tikznum{2}{3.5}{$\dots$}
            			
            			\tikzboxnum{1}{1}{$i$}
            			\tikzboxnum[\footnotesize]{1}{2}{$i+1$}
            			\tikzboxnum{2}{2}{$i$}
            			\tikzboxnum{1}{5}{$j$}
            			\tikzboxnum[\footnotesize]{2}{5}{$j-1$}
            		\end{tikzpicture}
            	};
            	\draw [->,decorate,decoration={snake,amplitude=.4mm}] (1a) -- (1b);
			}
		\]
	\item
		if $i-1\notin A_{a+1}$, $[i,j]\subset A_{a+1}$, 
		$[i,j]\subset A_{a}$, and $j+1\notin A_{a}$,
		then we conclude this decomposition does not give a reduced expression,
		reflecting the fact that
		$(s_{j}s_{j-1}\dots s_{i})(s_{j}\dots s_{i+1}s_{i})$
		is not a reduced expression.
		\[
			\tikzitem{
    			\node at (0,0)
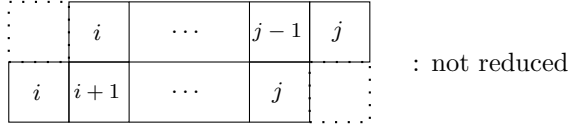
{
            		\begin{tikzpicture}[scale=0.8]
            			\tikzbox[thick, loosely dotted]{2}{1}
            			\tikzbox[thick, loosely dotted]{1}{6}
            			\draw (2,0) -- (4,0);
            			\draw (2,1) -- (4,1);
            			\draw (2,2) -- (4,2);
            			\tikznum{1}{3.5}{$\dots$}
            			\tikznum{2}{3.5}{$\dots$}
            			
            			\tikzboxnum{1}{1}{$i$}
            			\tikzboxnum[\footnotesize]{1}{2}{$i+1$}
            			\tikzboxnum{2}{2}{$i$}
            			\tikzboxnum{1}{5}{$j$}
            			\tikzboxnum[\footnotesize]{2}{5}{$j-1$}
            			\tikzboxnum{2}{6}{$j$}
            		\end{tikzpicture}
    			};
    			\node at (4,0) {: not reduced};
			}
		\]
\end{enumerate}

\noindent
Note that these moves look simpler when $i=j$:
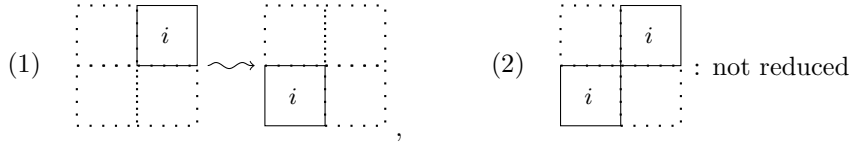
\[
	\tikzitem{
		\node at (-1.5,0) {(1)};
    	\node (1a) at (0,0){
    		\begin{tikzpicture}[scale=0.8]
    			\tikzbox[thick, loosely dotted]{1}{1}
    			\tikzbox[thick, loosely dotted]{2}{1}
    			\tikzbox[thick, loosely dotted]{1}{2}
    			\tikzboxnum{2}{2}{$i$}
    		\end{tikzpicture}
    	};
    	\node (1b) at (2.5,0){
    		\begin{tikzpicture}[scale=0.8]
    			\tikzbox[thick, loosely dotted]{2}{1}
    			\tikzbox[thick, loosely dotted]{1}{2}
    			\tikzbox[thick, loosely dotted]{2}{2}
    			\tikzboxnum{1}{1}{$i$}
    		\end{tikzpicture}
    	};
    	\draw [->,decorate,decoration={snake,amplitude=.4mm}] (1a) -- (1b);
	}
	,\hspace{10mm}
	\tikzitem{
		\node at (-1.5,0) {(2)};
		\node at (0,0){
		\begin{tikzpicture}[scale=0.8]
			\tikzbox[thick, loosely dotted]{2}{1}
			\tikzbox[thick, loosely dotted]{1}{2}
			\tikzboxnum{1}{1}{$i$}
			\tikzboxnum{2}{2}{$i$}
		\end{tikzpicture}
		};
		\node at (2,0) {: not reduced};
	}.
\]

It is shown in \cite[Section 3]{MR3001656} that,
for any decomposition
$w=d_{A_m}\dots d_{A_1}$ that gives a reduced expression,
we can apply a finite series of moves of type (1) to justify its diagram to the bottom
and obtain a $k$-code,
which is in fact uniquely determined from $w$ and denoted by $\RD(w)$,
and gives the
\textit{maximal decreasing decomposition} $w=d_{B_n}\dots d_{B_1}$,
that is, the vector $(|B_1|,\dots,|B_n|)$ is maximal in the lexicographical order
among such decompositions for $w$.
Furthermore,
this procedure bijectively maps
affine permutations to $k$-codes:
\begin{theo}[{\cite[Theorem 38]{MR3001656}}]\label{theo:RD_bij}
	The map $w\mapsto \mathrm{RD}(w)$ gives a bijection between
	$\aSn$ and the set of $k$-codes.
\end{theo}

\begin{exam}
Let $k=3$ and $w=s_{2}s_{30} s_{431}$ (this expression gives the maximal decreasing decomposition).
Then $\RD(w)=(0,2,0,1,3)$. (See Figure \ref{fig:k_code})
\end{exam}
\begin{figure}
\centering
\begin{tikzpicture}
	\node (RD) at (0,0){
		\begin{tikzpicture}[scale=0.4]
			\draw (0,0) -- (5,0);
			\tikzboxnum{1}{2}{1}
			\tikzboxnum{1}{4}{3}
			\tikzboxnum{1}{5}{4}
			\tikzboxnum{2}{2}{0}
			\tikzboxnum{2}{5}{3}
			\tikzboxnum{3}{5}{2}
		\end{tikzpicture}
	};
\end{tikzpicture}
\caption{$\RD(w)$ where $k=3$ and $w=s_{2}s_{30} s_{431}$}
\label{fig:k_code}
\end{figure}

\vspace{2mm}
Note that this construction also works if 
maximal decreasing decomposition is replaced with
\textit{maximal increasing decompositions}, that is,
the unique decomposition
$w=u_{B_n}\dots u_{B_1}$ into cyclically increasing elements
with the vector 
$(|B_1|,\dots,|B_n|)$ being maximal in the lexicographical order,
by modifying the notion of the filling of a $k$-code
so that the box in the $i$-th column and $j$-th row is marked with $j-i$ instead of $i-j$.
The resulting $k$-code is denoted by $\mathrm{RI}(w)$.
The map $w\mapsto\mathrm{RI}(w)$ also gives a bijection between $\aSn$ and the set of $k$-codes.

It is proved \cite[Corollary 39]{MR3001656} that
$w\in\aSn$ is $i$-dominant if and only if
the flattening of the corresponding $k$-code $\RD(w)$ forms a $k$-bounded partition
with residue $i$ in its lower left box,
that is,
$\RD(w)_i\ge\RD(w)_{i+1}\ge\dots\ge\RD(w)_{i-2}\ge\RD(w)_{i-1}=0$.
When $i=0$,
this mapping from 0-dominant permutations to $k$-bounded partitions coincides with
the one described earlier in Section \ref{sect:Prel::AffSym::Bijection}.
Moreover, it is proved \cite[Proposition 51]{MR3001656} that,
for $w\in\Wo$ the two corresponding $k$-codes $\RD(w)$ and $\RI(w)$,
regarded as $k$-bounded partitions, are transformed into each other by taking $k$-transpose: $\mathrm{sh}(\RI(w))=(\mathrm{sh}(\RD(w)))^{\omega_k}$
where $\mathrm{sh}(\alpha)\in\Pk$ is defined by
$\mathrm{sh}(\alpha)_j=|\{i \mid \alpha_i \ge j \}|$.

It is also proved in \cite[Proposition 56]{MR3001656} that
if $x\le_L y$ then
$\RD(x)\subset\RD(y)$ and $\RI(x)\subset\RI(y)$.

\begin{exam}
Let $k=3$ and $w=s_{0}s_1 s_{32} s_{03} s_{210} = s_{1} s_0 s_3 s_{12} s_{01} s_{30}$
(these presentations give the maximal decreasing and increasing decompositions).
Then $\RD(w)=(5,3,1,0)$ and
$\RI(w)=(6,3,0,0)$, and thus
$\mathrm{sh}(\RD(w)) = (3,2,2,1,1) = (2,2,2,1,1,1)^{\omega_3} = \mathrm{sh}(\RI(w))^{\omega_3}$.
(See Figure \ref{fig:RD_RI})
\end{exam}
\begin{figure}
\centering
\begin{tikzpicture}
	\node (RD) at (0,0) {
		\begin{tikzpicture}[scale=0.4]
			\tikzydiag{3,2,2,1,1}
			\tikzydnum{\small}{1}{0,1,2}
			\tikzydnum{\small}{2}{3,0}
			\tikzydnum{\small}{3}{2,3}
			\tikzydnum{\small}{4}{1}
			\tikzydnum{\small}{5}{0}
		\end{tikzpicture}
	};
	\node (cRD) at (3,0) {
		\begin{tikzpicture}[scale=0.4]
			\tikzyfill{3,1,1}
			\tikzydiag{6,3,3,1,1}
		\end{tikzpicture}
	};
	\node (cRI) at (6,0) {
		\begin{tikzpicture}[scale=0.4]
			\tikzyfill{3,1,1}
			\tikzydiag{5,3,3,1,1,1}
		\end{tikzpicture}
	};
	\node (RI) at (9,0) {
		\begin{tikzpicture}[scale=0.4]
			\tikzydiag{2,2,2,1,1,1}
			\tikzydnum{\small}{1}{0,3}
			\tikzydnum{\small}{2}{1,0}
			\tikzydnum{\small}{3}{2,1}
			\tikzydnum{\small}{4}{3}
			\tikzydnum{\small}{5}{0}
			\tikzydnum{\small}{6}{1}
		\end{tikzpicture}
	};
	\node [below of=RD, yshift=-5mm] {$\RD(w)$};
	\node [below of=RI, yshift=-5mm] {$\RI(w)$};
	\node [below of=cRD, yshift=-5mm] {$\core(\mathrm{sh}(\RD(w)))$};
	\node [below of=cRI, yshift=-5mm] {$\core(\mathrm{sh}(\RI(w)))$};
	\draw [|->] (RD) -- node[below]{$\core$} (cRD);
	\draw [|->] (RI) -- node[below]{$\core$} (cRI);
	\draw [<->] (cRD) -- node[below]{$(\cdot')$} (cRI);
\end{tikzpicture}
\caption{}
\label{fig:RD_RI}
\end{figure}

\subsubsection{$k$-rectangles}\label{sect:Prel::AffSym::k-rect}
%
%%% R_t
The partition
$(t^{k+1-t})=(t,t,\dots,t) \in\Pk$, for $1\le t \le k$,
is denoted by $R_t$
and called a {\it $k$-rectangle}.

\begin{rema}\label{theo:rem_Rt}
Consider the affine permutation $w_{R_i}$ corresponding to the $k$-rectangle $R_i$ under the bijection (\ref{eq:PCS}).
In fact $w_{R_i}$ is congruent,
in the extended affine Weyl group, 
to the translation $t_{-\varpi_i^\vee}$ 
by the negative of a fundamental coweight,
modulo left multiplication by the length zero elements.
\end{rema}

The next lemma describes the mapping $\la\mapsto\Rt\cup\la$ in terms of affine permutations.
For $A\subset I$ and $0\le t\le k$,
we write $A+t = \{a+t\mid a\in A\}$ ($\subset I$).

\begin{lemm}\label{theo:cupRt}
	Let $1\le t\le k$.
	Define a group isomorphism
	\[
		f_t \colon \aSn \lra \aSn \ ; \ s_i \mapsto s_{i+t} \quad \text{for $i\in I$}.
	\]
	For any $\la\in\Pk$, we have
	\[
		w_{\Rt\cup\la} = f_t(w_{\la}) w_{\Rt}.
	\]
\end{lemm}
\begin{proof}
	Let 
	$d_{A_m}\dots d_{A_1}$ and
	$d_{B_{k+1-t}}\dots d_{B_1}$ be
	the maximal decreasing decompositions of
	$w_\la$ and $w_{\Rt}$.
	Then
	$d_{A_m+t}\dots d_{A_1+t}$
	is the maximal decomposition of $f_t(w_\la)$.
	Stacking the $k$-code diagram of $f_t(w_\la)$ on that of $w_{\Rt}$,
	we obtain the diagram (not necessarily justified to the bottom)
	corresponding to the (not necessarily maximal) decreasing decomposition
	$f_t(w_\la) w_{\Rt} = d_{A_m+t}\dots d_{A_1+t} d_{B_{k+1-t}}\dots d_{B_1}$
	(See Figure \ref{fig:cupRt}).
	With maximizing moves,
	we can justify the diagram to
	obtain one with shape $\Rt\cup\la$,
	which corresponds to the maximal decomposition of $w_{\Rt\cup\la}$.
	
	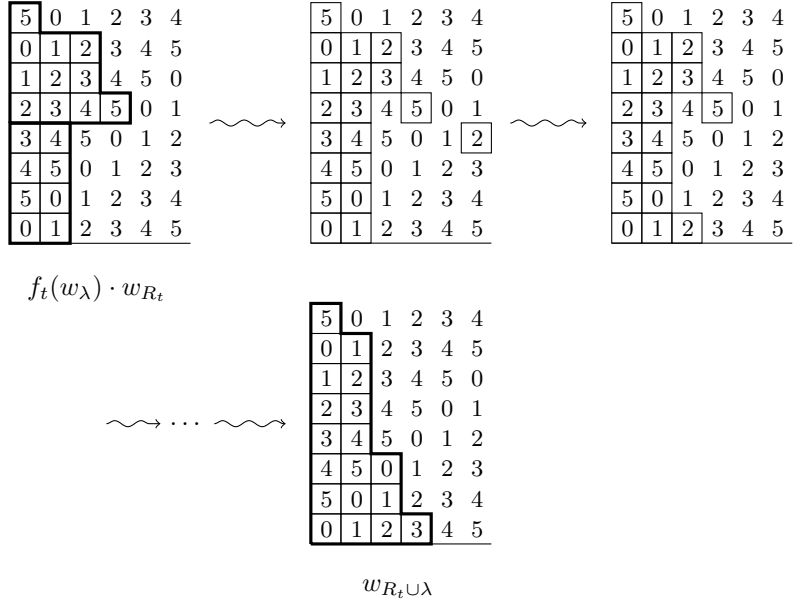
\begin{figure}
		\centering
    	\begin{tikzpicture}
        	\node (a) at (0,0){
        		\tikzitem[0.4]{
        			\draw (0,0) -- (6,0);
					\draw [very thick] (0,0) rectangle (2,4);
					\draw [very thick] (0,4) -| +(4,1) -| +(3,3) -| +(1,4) -| +(0,0);
					\tikzvnum{1}{0,1,2,3,4,5}
					\tikzvnum{2}{5,0,1,2,3,4}
					\tikzvnum{3}{4,5,0,1,2,3}
					\tikzvnum{4}{3,4,5,0,1,2}
					\tikzvnum{5}{2,3,4,5,0,1}
					\tikzvnum{6}{1,2,3,4,5,0}
					\tikzvnum{7}{0,1,2,3,4,5}
					\tikzvnum{8}{5,0,1,2,3,4}
					\tikzydiag{2,2,2,2,4,3,3,1}
				}
        	};
        	\node (b) at (4,0){
        		\tikzitem[0.4]{
        			\draw (0,0) -- (6,0);
					\tikzvnum{1}{0,1,2,3,4,5}
					\tikzvnum{2}{5,0,1,2,3,4}
					\tikzvnum{3}{4,5,0,1,2,3}
					\tikzvnum{4}{3,4,5,0,1,2}
					\tikzvnum{5}{2,3,4,5,0,1}
					\tikzvnum{6}{1,2,3,4,5,0}
					\tikzvnum{7}{0,1,2,3,4,5}
					\tikzvnum{8}{5,0,1,2,3,4}
					\tikzydiag{2,2,2,2,2,3,3,1}
					\tikzbox{4}{6}
					\tikzbox{5}{4}
				}
        	};
        	\node (c) at (8,0){
        		\tikzitem[0.4]{
        			\draw (0,0) -- (6,0);
					\tikzvnum{1}{0,1,2,3,4,5}
					\tikzvnum{2}{5,0,1,2,3,4}
					\tikzvnum{3}{4,5,0,1,2,3}
					\tikzvnum{4}{3,4,5,0,1,2}
					\tikzvnum{5}{2,3,4,5,0,1}
					\tikzvnum{6}{1,2,3,4,5,0}
					\tikzvnum{7}{0,1,2,3,4,5}
					\tikzvnum{8}{5,0,1,2,3,4}
					\tikzydiag{3,2,2,2,2,3,3,1}
					\tikzbox{5}{4}
				}
        	};
        	\node (d) at (4,-4){
        		\tikzitem[0.4]{
        			\draw (0,0) -- (6,0);
					\draw [very thick] (0,0) -| +(4,1) -| +(3,3) -| +(2,7) -| +(1,8) -| +(0,0);
					\tikzvnum{1}{0,1,2,3,4,5}
					\tikzvnum{2}{5,0,1,2,3,4}
					\tikzvnum{3}{4,5,0,1,2,3}
					\tikzvnum{4}{3,4,5,0,1,2}
					\tikzvnum{5}{2,3,4,5,0,1}
					\tikzvnum{6}{1,2,3,4,5,0}
					\tikzvnum{7}{0,1,2,3,4,5}
					\tikzvnum{8}{5,0,1,2,3,4}
					\tikzydiag{4,3,3,2,2,2,2,1}
				}
        	};
			\node (ccd) at (0,-4) {};
			\node (cd) at (1.2,-4) {$\dots$};
%			\node [below of=cRI, yshift=-5mm] {$\core(\mathrm{sh}(\RI(w)))$};
			\draw [->,decorate,decoration={snake,amplitude=.4mm}] (a) -- (b);
			\draw [->,decorate,decoration={snake,amplitude=.4mm}] (b) -- (c);
			\draw [->,decorate,decoration={snake,amplitude=.4mm}] (ccd) -- (cd);
			\draw [->,decorate,decoration={snake,amplitude=.4mm}] (cd) -- (d);
        	\node [below of=a,yshift=-12mm] {$f_t(w_\la) \cdot w_{\Rt}$};
        	\node [below of=d,yshift=-12mm] {$w_{\Rt\cup\la}$};
    	\end{tikzpicture}
		\caption{Justifying process with maximizing moves, 
			where $k=5$, $t=2$, $R_2=(2^4)$, and $\la=(4,3,3,1)$.}
    	\label{fig:cupRt}
	\end{figure}
\end{proof}

The next lemma explains
the correspondence between weak strips over $\la$ and weak strips over $\Rt\cup\la$.

\begin{lemm}\label{theo:Rt}
    Let $\la\in\Pk$.
    
    \noindent $(1)$
    For $A\subsetneq I$,
    if $d_A\la/\la$ is a weak strip then
    $\Rt\cup(d_{A}\la) = d_{A+t}(\Rt\cup\la)$.
    
    \vspace{1mm}
    \noindent
    Moreover, let $d_{A_1}\la, d_{A_2}\la, \dots$ 
    be the list of all weak strips over $\la$ (of size $r$).
    
    \noindent $(2)$
    $\Rt\cup(d_{A_1}\la), \Rt\cup(d_{A_2}\la), \dots$ 
    is the list of all weak strips over $\Rt\cup\la$ (of size $r$).
    
    \noindent $(3)$
    $d_{A_1+t}(\Rt\cup\la), d_{A_2+t}(\Rt\cup\la), \dots$ 
    is the list of all weak $r$-strips over $\Rt\cup\la$ (of size $r$).
\end{lemm}
\begin{proof}
	(2) is \cite[Theorem 20]{MR2079931}.
	(3) follows from (1) and (2).
	
	\noindent (1)
	It suffices to show the case $|A|=1$, that is,
	$\Rt\cup(s_i \la) = s_{i+t}(\Rt\cup\la)$.
	This is essentially shown in the process of proving \cite[Theorem 20]{MR2079931}
	by seeing correspondence between
	addable corners of $\core(\la)$ with residue $i$ and
	addable corners of $\core(\Rt\cup\la)$ with residue $i+t$,
	yet we here give another proof:
	by Lemma \ref{theo:cupRt}, it follows
	$w_{\Rt\cup(s_i \la)} = f_t(w_{s_i\la}) w_{\Rt}
		= f_t(s_i w_{\la}) w_{\Rt}
		= s_{i+t} f_t(w_{\la}) w_{\Rt}
		= s_{i+t} w_{\Rt\cup\la}$.
\end{proof}

\subsection{Symmetric functions}\label{sect:Prel::SymFcn}

For basic definitions for symmetric functions,
see for instance \cite[Chapter I]{MR1354144}.

\subsubsection{Symmetric functions}
Let $\La=\Z[h_1,h_2,\dots]$ be the ring of symmetric functions,
generated by the 
\textit{complete symmetric functions}
$ h_r = \sum_{i_1\le i_2\le\dots\le i_r} x_{i_1}\dots x_{i_r}$.
%
%%% h_la
For a partition $\la$
we set 
$h_\la = h_{\la_1}h_{\la_2}\dots h_{\la_{l(\la)}}$.
The set
$\{h_\la\}_{\la\in \mathcal{P}}$
forms a $\Z$-basis of $\La$.

\subsubsection{Schur functions}
%%% Pieri
The \textit{Schur functions} $\{s_\la\}_{\la\in\mathcal{P}}$ are the family of symmetric functions satisfying the \textit{Pieri rule}:
\[ h_r s_\la = \sum_{\text{$\mu/\la$:horizontal strip of size $r$}} s_\mu.\]

\subsubsection{$k$-Schur functions}
We recall a characterization of $k$-Schur functions given in \cite{MR2331242},
since it is a model for and has a relationship with $K$-$k$-Schur functions.

\begin{defi}[$k$-Schur function via $k$-Pieri rule]
$k$-Schur functions
$\{\ks{w}\}_{w\in\aG}$ are the family of symmetric functions
such that 
\begin{align*}
	\ks{e}&=1, \\
	h_r \ks{w} &= \sum_{v} \ks{v} \quad \text{for $1\le r\le k$ and $w\in\aG$,}
\end{align*}
summed over $v\in\aG$ such that 
$v/w$ is a weak strip of size $r$.
\end{defi}

It is known that $\{\ks{w}\}_{w\in\aG}$ forms a basis of $\Lk=\Z[h_1,\dots,h_k]\subset\La$,
and $\ks{w}$ is homogeneous of degree $l(w)$.
We regard $\ks{\la}$ as $\ks{w_\la}$ for $\la\in\Pk$.
It is proved in \cite[Theorem 40]{MR2331242} that
\begin{prop}[$k$-rectangle property]\label{prop:kS_fac}
For $1\le t \le k$ and $\la\in\Pk$, we have
$
	\ks{R_t\cup\la} = \ks{R_t} \ks{\la} (=s_{R_t} \ks{\la}).
$
\end{prop}

\subsubsection{$K$-$k$-Schur functions}

In this paper we employ the following characterization with the Pieri rule
(\cite[Corollary 7.6]{MR2660675}, \cite[Corollary 50]{Morse12}) of the $K$-$k$-Schur function as its definition.
\begin{defi}[$K$-$k$-Schur function via $K$-$k$-Pieri rule]\label{defi:KkPieri}
$K$-$k$-Schur functions $\{\kks{w}\}_{w\in\aG}$ are the family of symmetric functions such that 
$\kks{e}=1$ and 
\[
h_r \cdot \kks{w} =
\sum_{(A,v)} (-1)^{r+l(w)-l(v)} \kks{v},
	\]
for $w \in \aG$ and $1 \le r \le k$,
summed over $v\in\aG$ and $A\subsetneq I$ such that 
$(v/w,A)$ is an affine set-valued strip of size $r$.
\end{defi}

It is known that $\{\kks{w}\}_{w\in\aG}$ forms a basis of $\Lk$.
Besides,
though $\kks{w}$ is an inhomogeneous symmetric function in general,
the degree of $\kks{w}$ is $l(w)$
and its homogeneous part of highest degree is equal to $\ks{w}$.
In this paper, for $f=\sum_{w} c_w \kks{w} \in \Lk$ we write $[\kks{v}](f)=c_v$.

\section{Properties on the strong and weak orderings on Coxeter groups}\label{sect:PropOrder}

In this section we let $(W,S)$ be an arbitrary Coxeter group.

Recall that for a poset $(P,\le)$ and a subset $A\subset P$,
if the set $\{z\in P\mid z \le y \text{ for any }y\in A\}$ has the maximum element $z_0$
then $z_0$ is called the \textit{meet} of $A$ and denoted by $\bigwedge A$,
and if $\{z\in P\mid z \ge y\text{ for any }y\in A\}$ has the minimum element
then it is called the \textit{join} of $A$ and denoted by $\bigvee A$.
When $A=\{x,y\}$, its meet and join are simply called the meet and join of $x$ and $y$,
and denoted by $x\wedge y$ and $x\vee y$.
A poset for which any nonempty subset has the meet is called a \textit{complete meet-semilattice}.
A poset for which any two elements have the meet and join is called a \textit{lattice}.
A subset of a complete meet-semilattice has the join if it has a common upper bound, since the join is the meet of all its common upper bounds then.

In this paper we denote the meet of $x,y\in W$ under
the strong (resp.\,left, right) order
by $x \wedge y$ (resp.\,$x\wedge_L y$, $x\wedge_R y$) and 
call it the \textit{strong meet} (resp.\,\textit{left meet}, \textit{right meet})
of $\{x,y\}$.
We define $x\vee y$, $x\vee_L y$ and $x\vee_R y$ similarly.

\subsection{Lattice property of the weak order}\label{sect:PropOrder::Lattice}

It is known that the weak order on any Coxeter group or its parabolic quotient forms complete meet semilattice (see, for example, \cite[Theorem 3.2.1]{MR2133266}).
The join of two elements in them, however, does not always exist,
but it is known that the quotient of an affine Weyl group by its corresponding finite Weyl group forms a lattice under the weak order \cite{MR1740744}.
We here include another proof for the type affine A case for the sake of completeness.

\begin{lemm}\label{theo:join_exist}
	For any $v,w\in \aG$, their join $v\vee_L w$ under the left weak order exists.
\end{lemm}
\begin{proof}
	Since $\aG$ is a meet complete semilattice,
	it remains to show the existence of a common upper bound of $v$ and $w$ under the left order.
	Let $\ks{v}$ and $\ks{w}$ denote the $k$-Schur functions corresponding to $v$ and $w$.
	In the expansion of their product in the $k$-Schur function basis
	$\ks{v}\ks{w}=\sum_{u} c^{u}_{vw} \ks{u}$,
	every $u$ appearing in the right-hand side satisfies $w\le_L u$
	because $\ks{v}$ can be written as a polynomial in $h_1,\dots,h_k$
	and by the Pieri rule $h_i \ks{x}$ is in general a linear combination of $\ks{y}$ with $y\ge_L x$.
	By the same reason we have $v\le_L u$.
\end{proof}

We proved the following corollary in the proof of the lemma above:
\begin{coro}\label{theo:split_condition_with_join}
	For any $v,w\in \aG$,
	every $u$ appearing with a nonzero coefficient in the right-hand side of $\ks{v}\ks{w}=\sum_{u} c^{u}_{vw} \ks{u}$ satisfies
	$u \ge_L v\vee_L w$.
\end{coro}

With the $K$-$k$-Pieri rule instead of the $k$-Pieri in hand, the same holds for the $K$-$k$-Schur functions:
\begin{coro}\label{theo:Kk_split_condition_with_join}
	For any $v,w\in \aG$,
	every $u$ appearing with a nonzero coefficient in the right-hand side of $\kks{v}\kks{w}=\sum_{u} d^{u}_{vw} \kks{u}$ satisfies
	$u \ge_L v\vee_L w$.
\end{coro}

\subsection{Properties of Demazure and anti-Demazure actions} \label{sect:PropOrder::phipsi}

\begin{lemm}\label{theo:b_i}
	Let $x\in W$ and 
	$\phi_x,\psi_x$ be the Demazure and anti-Demazure actions defined in Section \ref{sect:Prel::Coxeter::0Hecke}.
	\begin{enumerate}[label=\textup{(\arabic*)}]
		\item
			$\phi_x(w) \ge_L w$ and $\psi_x(w)\le_L w$ for any $w\in W$.
		\item
			$\phi_x$ and $\psi_x$ are order-preserving under $\le$.
			Namely,
			if $v\le w$ then $\phi_x(v)\le \phi_x(w)$ and $\psi_x(v)\le\psi_x(w)$.
		\item
			For any $y\in W$, the map $(x\mapsto \phi_x(y))$ is order-preserving
			and the map $(x\mapsto \psi_x(y))$ is order-reversing
			under $\le$.
		\item
			$\phi_x \psi_{x^{-1}}(y) \ge y$ and
			$\psi_{x^{-1}} \phi_x(y) \le y$ for any $y\in W$.
		\item
			$\phi_x$ preserves strong meets and
			$\psi_x$ preserves strong joins.
			Namely,
			for $v, w\in W$,
			\begin{enumerate}[label=\textup{(}\alph*\textup{)}]
				\item
					if $v\wedge w$ exists then $\phi_x(v)\wedge \phi_x(w)$ exists and equals to $\phi_x(v\wedge w)$.
				\item
					if $v\vee w$ exists then $\psi_x(v)\vee \psi_x(w)$ exists and equals to $\psi_x(v\vee w)$.
			\end{enumerate}
	\end{enumerate}
\end{lemm}

\begin{rema}
This lemma also works for $\phi^R_x$ and $\psi^R_x$ instead of $\phi_x$ and $\psi_x$.
\end{rema}

\begin{rema}
	For the statements on $\phi_x$, 
	(1) of this lemma is done in \cite[Proposition 3.1(d)]{buch2015};
	(2) and (3) in \cite[Proposition 3.1(c)]{buch2015}.
\end{rema}

\begin{proof}
	(1) is clear from the definition of $\phi_s$ and $\psi_s$.
	(2) is from the Lifting Property.
	(3) is clear from (1) and the Subword Property.
	(4)
	The case $x=s\in S$ is clear from the definition of $\phi_s,\psi_s$,
	and the general case follows from this and (2).
	
	\noindent
	(5)
	($a$)
	It suffices to prove it when $x=s\in S$.
	Write simply $\phi = \phi_s$ and $\psi=\psi_s$.
	Assume $v\wedge w$ exists.
	We have $\phi(v\wedge w) \le \phi(v), \phi(w)$ by (2).
	To show that $\phi(v\wedge w)$ is the meet of $\phi(v)$ and $\phi(w)$,
	take arbitrary $u$ with $u\le \phi(v), \phi(w)$.
	Then $\psi(u)\le \psi(v), \psi(w)$ from the Lifting Property,
	and hence $\psi(u)\le v, w$, 
	which implies $\psi(u)\le v\wedge w$.
	Applying $\phi$, we have
	$\phi(u) = \phi(\psi(u)) \le \phi(v\wedge w)$,
	and hence $u\le \phi(v\wedge w)$.
	($b$) is essentially the same as ($a$).
\end{proof}

\begin{rema}
The map $\phi_x$ (resp.\,$\psi_x$) does not preserve strong joins (resp.\,meets) in general.
For example,
letting $W=S_4$,
we have $s_{212}\wedge s_{232} = s_{2}$ but
$\psi_2(s_{212})\wedge \psi_2(s_{232}) = s_{12}\wedge s_{32} = s_{2} \neq \psi_2(s_2)$,
where we write $s_{ij\dots}$ instead of $s_i s_j \cdots$.
Mapping everything above via $x\mapsto xw_0$ where $w_0$ is the longest element of $W$,
we obtain a counterexample for $\phi_x$ preserving joins.
\end{rema}

\begin{coro}\label{theo:ux=vy}
	Let $u,v,x,y\in W$ with
	$\redprod{u}{x}$ and $\redprod{v}{y}$ are reduced
	and $ux=vy$
	(or namely, $u\le_L ux=vy\ge_L v$).
	Then $u\ge v \iff x\le y$.
\end{coro}
\begin{proof}
	By Lemma \ref{theo:b_i}(3) we have
	$u\ge v \iff u^{-1}\ge v^{-1} \implies (x=)\,\psi_{u^{-1}}(ux)\le\psi_{v^{-1}}(vy)\,(=y)$.
	The other direction is similar.
\end{proof}

\subsection{Half-strong, half-weak meets and joins}
\label{sect:PropOrder::SLjoin}

Analogous to the meets and joins under the weak order,
we show the existence of the minimum element (under $\le$) of the set
\begin{gather*}
	\{z \in W \mid x\le z \ge_L y\},
\end{gather*}
and the maximum of
\begin{gather*}
	\{z \in W \mid x\ge_L z \le y\}.
\end{gather*}

\begin{rema}
	It seems that the existence of such elements has been known;
	for example, in his Sage implementation to compute the Deodhar lift \cite{Deodhar87},
	Shimozono explicitly used (1) of the following proposition.
	However we do not know about a reference,
	so we take the opportunity to give one here.
	The proof of (1) of the following proposition is by Shimozono \cite{PC_Shimozono}.
\end{rema}

\begin{prop}\label{theo:SL_join'}
	Let $x,y\in W$.
	\begin{enumerate}[label=\textup{(\arabic*)}]
		\item
			The set $\{u\in W\mid x\le \phi_u(y)\}$ has the minimum element
			$\psi^R_{y^{-1}}(x)$ under the strong order.
		\item
			The set $\{u\in W\mid \psi_{u^{-1}}(x)\le y\}$ has the minimum element
			$\psi^R_{y^{-1}}(x)$ under the strong order.
	\end{enumerate}
\end{prop}
\begin{proof}
	\noindent
	(1)
	We prove it by induction on $l(y)$.
	The base case $l(y)=0$ being clear,
	we assume $l(y)>0$.
	Take $s\in S$ such that $y>ys$.
	Let $x'=\psi^R_{s}(x)\,(=\min(x,xs))$ and $y'=ys$.
	Since $y=y'*s$, for any $u$ we see $u*y=u*y'*s$,
	whence by the Lifting Property $x\le u*y \iff x'\le u*y'$.
	Hence $D(x,y)=D(x',y')$ where $D(x,y) = \{u\in W\mid x\le \phi_u(y)\,(=u*y)\}$.
	By the induction hypothesis it follows that
	$D(x,y)=D(x',y')$ has the minimum element $\psi^R_{y'^{-1}}(x')$,
	which equals to $\psi^R_{y^{-1}}(x)$.
	
	\noindent
	(2)
	Let $E(x,y)=\{u\in W\mid \psi_{u^{-1}}(x)\le y\}$.
	It suffices to show $D(x,y)=E(x,y)$.
	By Lemma \ref{theo:b_i}(2),(4) we have
	$x\le \phi_u(y) \implies \psi_{u^{-1}}(x)\le\psi_{u^{-1}}\phi_u(y) \le y$ and
	$\psi_{u^{-1}}(x)\le y \implies x\le\phi_u\psi_{u^{-1}}(x)\le\phi_u(y)$.
\end{proof}

%%% SL join
%
\begin{prop}\label{theo:SL_join}
	Let $x,y\in W$.
	\begin{enumerate}[label=\textup{(\arabic*)}]
		\item
			The set $\{z\in W\mid x\le z\ge_L y\}$ has the minimum element
			$\psi^R_{y^{-1}}(x)y$ under the strong order.
		\item
			The set $\{z\in W\mid x\ge_L z\le y\}$ has the maximum element
			$\big(\psi^R_{y^{-1}}(x)\big)^{-1}x$ under the strong order.
	\end{enumerate}
\end{prop}
\begin{proof}
	\noindent{(1)}
	By (\ref{eq:gq_isom}), we have
	$D(x,y)\supset \{u\mid x\le uy\ge_L y\} \simeq \{z\mid x\le z\ge_L y\}$;
	$u\mapsto uy$,
	where the isomorphism is under $\le$.
	The minimum element $u$ of $D(x,y)$ satisfies $u*y=uy$ i.e.\,$uy\ge_Ly$,
	since otherwise $(u*y)y^{-1}$ is a smaller element of $D(x,y)$.
	Hence by Proposition \ref{theo:SL_join'}(1)
	we have $\psi^R_{y^{-1}}(x)y =\min_{\le}\{z\mid x\le z\ge_L y\}$.
	
	\noindent{(2)}
	By Corollary \ref{theo:ux=vy} we have
	$E(x,y)\supset \{u\mid x\ge_L u^{-1}x\le y\} \underset{\text{anti}}{\simeq} \{z\mid x\ge_L z\le y\}$;
	$u\mapsto u^{-1}x$,
	where the anti-isomorphism is under $\le$.
	For a similar reason to (1) we have 
	$\max_{\le}\{z\mid x\ge_L z\le y\} = (\min_\le E(x,y))^{-1}x 
		= (\psi^R_{y^{-1}}(x))^{-1}x$.
\end{proof}

From the proposition above, we define 
\begin{align*}
	x\SLjoin y = y\LSjoin x 
		&:= \min_{\le}\{z\in W\mid x\le z\ge_L y\} = \psi^R_{y^{-1}}(x)y, \\
	x\LSmeet y = y\SLmeet x 
		&:= \max_{\le}\{z\in W\mid x\ge_L z\le y\} = \big(\psi^R_{y^{-1}}(x)\big)^{-1}x.
\end{align*}
We define $x\SRjoin y$ and $x\SRmeet y$ similarly.

\subsection{Flipping lower weak intervals}
\label{sect:PropOrder::flip}

For any $z\in W$,
define the map
\[\Phi_z: [e, z]_L \longrightarrow [e, z]_R; x \mapsto zx^{-1}\]
and its inverse
\[\Psi_z: [e, z]_R \longrightarrow [e, z]_L; y \mapsto y^{-1}z.\]

Proposition \ref{theo:anti_isom} below demonstrates that
these maps behave well along with the strong order on $W$ and its meet/join operations.
\Todo{where it is needed?}

\begin{prop}\label{theo:anti_isom}
	Let $z\in W$.
	\begin{enumerate}[label=\textup{(\arabic*)}]
		\item
			$\Phi_z$ and $\Psi_z$ are anti-isomorphisms under the strong order.
		\item
			$l(\Phi_z(x)) = l(z)-l(x)$ for any $x\in[e,z]_L$ and
			$l(\Psi_z(y)) = l(z)-l(y)$ for any $y\in[e,z]_R$.
		\item
			$\Phi_z$ and $\Psi_z$ send strong meets to strong joins.
			Namely, 
			\begin{enumerate}[label=\textup{(}\alph*\textup{)}]
				\item
					for $x, y\in [e, z]_L$ such that $x\wedge y$ exists and $x\wedge y \in [e,z]_L$,
					we have $\Phi_z(x\wedge y)=\Phi_z(x)\vee\Phi_z(y)$.
				\item
					for $x, y\in [e, z]_R$ such that $x\wedge y$ exists and $x\wedge y \in [e,z]_R$,
					we have $\Psi_z(x\wedge y)=\Psi_z(x)\vee\Psi_z(y)$.
			\end{enumerate}
			(Note that the meets and joins are not taken in $[e,z]_L$ or $[e,z]_R$ but in $W$.)
	\end{enumerate}
\end{prop}

\begin{proof}
	(1) is done in Corollary \ref{theo:ux=vy}, and (2) is obvious.
	
	\noindent{(3)}
	We only prove ($a$) since ($b$) is shown similarly.
	
	Let $x, y, x\wedge y\in[e, z]_L$.
	From (1) it follows that $\Phi_z(x\wedge y)\ge \Phi_z(x),\Phi_z(y)$.
	To show the minimality of $\Phi_z(x\wedge y)$,
	let us take arbitrary $w\in W$ such that $w\ge \Phi_z(x), \Phi_z(y)$.
	From Proposition \ref{theo:SL_join}, we can let $w'=z\RSmeet w$.
	Since $\Phi_z(x),\Phi_z(y)\in [e,z]_R\cap [e,w]$,
	we have $\Phi_z(x),\Phi_z(y)\le w'$.
	Since $w'\le_R z$, applying $\Psi_z$ ($=\Phi_z^{-1}$),
	we have $x,y\ge \Psi_z(w')$.
	Hence $x\wedge y \ge \Psi_z(w')$.
	Applying $\Phi_z$, we have $\Phi_z(x\wedge y) \le w'$,
	and hence $\Phi_z(x\wedge y)\le w$.
	Therefore $\Phi_z(x\wedge y)$ is the join of $\{\Phi_z(x), \Phi_z(y)\}$.
\end{proof}

\begin{rema}
	It seems to be true that $\Phi_z$ and $\Psi_z$ send strong joins to strong meets.
	Its proof would require that there be the strong-minimum element of
	$\{z\mid x\le z\le_L y\}$ and
	the strong-maximum of $\{z\mid x\le_L z\le y\}$
	for any $x,y \in W$, 
	analogous to Proposition \ref{theo:SL_join}.
\end{rema}

\subsection{Chain Property for lower weak intervals}\label{sect:PropOrder::CP}

In this section we prove the Chain Property for the lower weak intervals
$[e,u]_L$ and $[e,u]_R$ for arbitrary Coxeter group $W$ and its element $u\in W$.
\Todo{where it is needed?}
This is similar to that for the generalized quotients,
in that 
$[e,u]_L = \{x\mid x\le_L u\}$
whereas
$W/\{u\} \simeq \{x\mid x\ge_L u\}$.
Besides
it is shown in \cite[Corollary 4.5]{MR946427} that 
the class of right generalized quotients and lower left intervals coincide for finite $W$.
When $W$ is infinite, however, these do not,
as we give a counterexample below.
Beforehand we recall
\cite[Theorem 4.10]{MR946427}:
for any Coxeter group $W$,
the left generalized quotients and the right generalized quotients are in bijection by
$U\mapsto W/U$ and $V\backslash W \mapsfrom V$,
and
a subset $U\subset W$ is a right generalized quotient if and only if $U=W/(U\backslash W)$.

\begin{exam}
Let $W=\aSn=\langle s_0,s_1,\dots,s_{k}\rangle$.
Let $w_0$ be the longest element of $S_{k+1}=\langle s_1,\dots,s_{k}\rangle$.
From the following claim
we have
$s_0w_0\in \aSn/(S_{k+1}\backslash \aSn)$,
and thereby $S_{k+1}=[e,w_0]_L$ is not a right generalized quotient of $\aSn$.
\end{exam}
%\vspace{2mm}
\noindent\textit{Claim.}
For any $z\in \aSn$, the product
$\redprod{w_0}{z}$ is reduced if and only if $\redprod{s_0w_0}{z}$ is reduced.

\begin{proof}[Proof of Claim]
	The ``if'' direction is clear.
	Toward the ``only if'' direction,
	assume $\redprod{w_0}{z}$ is reduced, that is, $\redprod{z^{-1}}{w_0}$ is reduced.
	Since $z^{-1}w_0\ge_L w_0$,
	we have $\mathrm{RD}(z^{-1}w_0)\supset \mathrm{RD}(w_0)$.
	Hence, since the first row of $\mathrm{RD}(w_0)$ is $\{1,\dots,k\}$ and the rows of a $k$-code are proper subsets of $\{0,1,\dots,k\}$,
	the first row of $\mathrm{RD}(z^{-1}w_0)$ is also $\{1,\dots,k\}$.
	Thus, inserting $s_0$ into $\mathrm{RD}(z^{-1}w_0)$ from the bottom and justifying it to the bottom with maximizing moves,
	we successfully obtain $\mathrm{RD}(z^{-1}w_0s_0)$, the $i$-th column of which is
	\begin{itemize}
		\item
			the $k$-th column of $\mathrm{RD}(z^{-1}w_0)$ with an $s_0$ added, when $i=0$,
		\item
			the $i$-th column of $\mathrm{RD}(z^{-1}w_0)$ when $i=1,\dots,k-1$,
		\item
			empty when $i=k$.
	\end{itemize}
	(See Figure \ref{fig:RD_justify})
	In particular $\redprod{z^{-1}w_0}{s_0}$ is reduced.
	Combining this with that $\redprod{z^{-1}}{w_0}$ is reduced, 
	we have $\redprodd{z^{-1}}{w_0}{s_0}$ is reduced,
	and hence so is $\redprodd{s_0}{w_0}{z}$,
	as desired.
	
	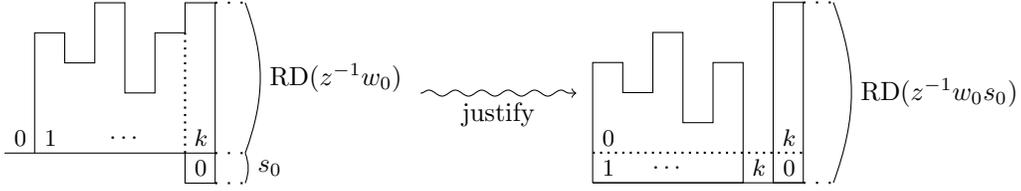
\begin{figure}
		\centering
		\begin{tikzpicture}
			\node (before) at(0,0) {
				\begin{tikzpicture}[scale=0.4]
					\draw (0,0) -- (7,0);
					\draw (1,0) |- (2,4) |- (3,3) |- (4,5) |- (5,2) 
						|- (6,4) |- (7,5) |- (6,-1) -- (6,0);
					\draw [thick, dotted] (6,0) -- (6,4);
					\tikznum{1}{1}{0};
					\tikznum{1}{2}{1};
					\tikznum{1}{7}{$k$};
					\tikznum{1}{4.5}{$\dots$};
					\tikznum{0}{7}{0};
					
					\draw [thick, loosely dotted] (7,0) -- (8,0);
					\draw [thick, loosely dotted] (7,5) -- (8,5);
					\draw [thick, loosely dotted] (7,-1) -- (8,-1);
					\draw (8,0) to [out=70, in=-70] node[right]{$\RD(z^{-1}w_0)$} (8,5);
					\draw (8,-1) to [out=70, in=-70] node[right]{$s_0$} (8,0);
				\end{tikzpicture}
			};
			\node (after) at(8,0) {
				\begin{tikzpicture}[scale=0.4]
					\draw [dotted, thick] (0,0) -- (7,0);
					\draw (0,-1) -- (7,-1);
					\draw (0,-1) |- (1,3) |- (2,2) |- (3,4) |- (4,1) |- (5,3) |- (0,-1);
					\draw (6,-1) rectangle (7,5);
					\tikznum{1}{1}{0};
%					\tikznum{1}{5}{$k-2$};
%					\tikznum{1}{6}{$k-1$};
					\tikznum{1}{7}{$k$};
					\tikznum{0}{3}{$\dots$};
					\tikznum{0}{1}{1};
%					\tikznum[\scriptsize]{0}{5}{$k-1$};
					\tikznum{0}{6}{$k$};
					\tikznum{0}{7}{0};					

					\draw [thick, loosely dotted] (7,5) -- (8,5);
					\draw [thick, loosely dotted] (7,-1) -- (8,-1);
					\draw (8,-1) to [out=70, in=-70] 
						node[right]{$\RD(z^{-1}w_0s_0)$} (8,5);
				\end{tikzpicture}
			};
			\draw [->, decorate,decoration={snake,amplitude=.4mm}] 
				(before) -- node[below] {justify} (after);
		\end{tikzpicture}
		\caption{Inserting $s_0$ into $\RD(z^{-1}w_0)$ and justify it to obtain a $k$-code for $z^{-1}w_0s_0$}
		\label{fig:RD_justify}
	\end{figure}
\end{proof}

The proof of the following proposition is parallel to
that of \cite[Theorem 3.4]{MR946427}.
Beforehand we recall that, for $x,y\in W$ with $x\ge y$ and
any fixed reduced expression $x=s_{1}\dots s_{m}$,
there exists $1\le j_1<j_2<\dots<j_l\le m$ such that
$x=y^{(0)}\gecov y^{(1)}\gecov\dots\gecov y^{(l)}=y$ where
\[
	y^{(a)} = s_{1}\dots \widehat{s_{j_1}} \dots \widehat{s_{j_a}} \dots s_{m}.
\]
See, for example, \cite[Section 3]{MR946427} or \cite{MR644668} for the detail.

\begin{prop}\label{theo:ChainProp_weakinterval}
	Let $u,x,y\in W$ with
	$xu,yu\le_L u$ and $xu\le yu$.
	Note that
	$xu\le yu \iff x^{-1}\ge y^{-1} \iff x\ge y$ for $xu,yu\le_L u$.
	Fix a reduced expression for $x=s_{1}\dots s_{m}$
	and take $y^{(0)},\dots,y^{(l)}$ as right above.
	Then $y^{(a)}u\le_L u$ for any $a$.
\end{prop}

\begin{proof}
	Suppose to the contrary that there exists $a$ such that $y^{(a)}u\not\le_L u$.
	Since $y^{(l)}u=yu\le_L u$,
	we can take such $a$ that $y^{(a)}u\not\le_L u$ and $y^{(a+1)}u\le_L u$.
	
	Since $xu\le_L u$, we have
	$s_{j_a+1}\dots s_{m}u\le_L u$.
	Hence there exists $p<j_a$ such that
	\begin{equation}\label{eq:sipz}
		s_{p}zu \not\le_L u \quad\text{and}\quad zu \le_L u,
	\end{equation}
	where we put
	\[
		z = s_{p+1}\dots \widehat{s_{j_a}} \dots s_{j_{a+1}} \dots s_{m},
	\]
	where there may be more indices omitted between
	$s_{p+1}$ and $\widehat{s_{j_1}}$,
	according to the omissions in
	$y^{(a)} = s_{1}\dots \widehat{s_{j_1}} \dots \widehat{s_{j_a}} \dots s_{m}$.
	Since $y^{(a+1)}u\le_L u$, we have
	\begin{equation}\label{eq:sipzhat}
		s_{p}\widehat{z}u \le_L u \quad\text{and}\quad \widehat{z}u \le_L u,
	\end{equation}
	where we put
	\[
		\widehat{z} = s_{p+1}\dots \widehat{s_{j_a}} \dots \widehat{s_{j_{a+1}}} \dots s_{m}.
	\]
	
	We have
	$zu\lecov s_{p}zu$ by (\ref{eq:sipz})
	and
	$\widehat{z}u\gecov s_{p}\widehat{z}u$ by (\ref{eq:sipzhat}).
	Besides, since $y^{(a)}\gecov y^{(a+1)}$ it follows $z\gecov\widehat{z}$, 
	and thereby $zu\lecov\widehat{z}u$.
	Hence we have $s_{p}zu=\widehat{z}u$ by the Lifting Property and length arguments.
	Therefore $s_{p}z=\widehat{z}\lecov z$,
	which contradicts the fact that $s_{p}z$ is a consecutive subword of a reduced expression for $y^{(a)}$.
\end{proof}

As a corollary, we have the Chain Property for weak lower intervals:
\begin{theo}\label{theo:CP_weakorderideal}
	For any $u\in W$,
	the principal order ideal 
	$[e,u]_L$ (resp.\,$[e,u]_R$)
	under the left (resp.\,right) weak order has the Chain Property.
\end{theo}
\begin{proof}
	The statement for left lower intervals follows from
	Proposition \ref{theo:ChainProp_weakinterval} and that
	$\{x\in W\mid xu\le_L u\} = [e,u^{-1}]_L$
	for $u\in W$,
	which follows from
	$xu\le_L u \iff x^{-1}\le_R u \iff x\le_L u^{-1}$.
	The statement for right intervals is proved parallely.
\end{proof}

\section{Properties on the weak strips}\label{sect:PropWS}

Hereafter, throughout this paper,
we restrict our attention to $\aSn$ rather than general Coxeter groups
and let $W=\aSn$ and $\Wo=\aG$.
In Section \ref{sect:Prel::AffSym}
we put $I=\Z_{k+1}=\{0,1,\dots,k\}$ and
let $d_A$ denote the cyclically decreasing element corresponding to $A\subsetneq I$.

In this section we prove some properties on weak strips.
First we define for any $u\in W$,
\Todo{$D=\dots$}
\begin{align*}
	Z_{u,+} 
		&= \{v\in W \mid v=d_{A}u\ge_L u \text{ for $\exists A\subsetneq I$}\}, \\
	Z'_{u,+} 
		&= \{A\subsetneq I \mid d_{A}u\ge_L u \}
		 = \{A\subsetneq I \mid d_{A}u\in Z_{u,+} \}, \\
	Z_{u,-} 
		&= \{v\in W \mid v=d_{A}^{-1}u\le_L u \text{ for $\exists A\subsetneq I$}\}, \\
	Z'_{u,-} 
		&= \{A\subsetneq I \mid d_{A}^{-1}u\le_L u \}
		 = \{A\subsetneq I \mid d_{A}^{-1}u\in Z_{u,-} \}.
\end{align*}
It is an immediate observation from the Subword Property that
\begin{itemize}
	\item
		The map $(Z'_{u,+}, \subset) \lra (Z_{u,+}, \le)\,;\, A\mapsto d_Au$
		is an isomorphism of posets.
	\item
		The map $(Z'_{u,-}, \subset) \lra (Z_{u,-}, \le)\,;\, A\mapsto d_A^{-1}u$
		is an anti-isomorphism of posets.
\end{itemize}

Since if $u\in\Wo$ and $v\le_L u$ then $v\in\Wo$, for $u\in\Wo$ we have
\[
	Z_{u,-} = \{v \mid \text{$u/v$ is a weak strip}\}.
\]
On the other hand, the set $Z_{u,+}$ does not coincide with the set of $v$ such that 
$v/u$ is a weak strip. More precisely, for $u\in\Wo$ we have by definition
\begin{align*}
	\text{$v/u$ is a weak strip}
	&\iff v\in Z_{u,+} \text{ and } v\in\Wo.
	\intertext{
		Recalling that $v\in\Wo \iff vw_0^J \ge_L w_0^J$ where
		$J=\{1,\dots,k\}$ and $w_0^J$ is the longest element of $W_J=S_{k+1}$,
		by Lemma \ref{theo:weakorder_xyz} we have
	}
	\text{$v/u$ is a weak strip}
	&\iff vw_0^J \in Z_{uw_0^J,+}  \\
	&\iff v=d_A u \text{ with } A\in Z'_{uw_0^J,+}.
\end{align*}
In other words, defining
\begin{align*}
	Z_{u,+}^\circ 
		&= \{v \mid \text{$v/u$ is a weak strip}\}, \\
	Z_{u,+}^{'\circ}
		&= \{A\subsetneq I \mid \text{$d_{A}u/u$ is a weak strip} \}
		 = \{A\subsetneq I \mid d_{A}u\in Z_{u,+}^\circ \},
\end{align*}
we have
\begin{align*}
	Z_{u,+}^\circ &\simeq Z_{uw_0^J,+}\ ;\ v \mapsto v w_0^J, \\
	Z_{u,+}^{'\circ} &= Z'_{uw_0^J,+}.
\end{align*}

\begin{exam}
Figure \ref{fig:ZandZ'} illustrates the same example as Example \ref{exam:ws}.
\end{exam}
\begin{figure}
	\centering
	\begin{tikzpicture}
    	\node (Z) at (0,0) {
        	\begin{tikzpicture}[scale=0.7]
        		\node (la) at (0,0) {$\la$};
        		\node (1) at (-2,2) {$s_1\la$};
        		\node (3) at (2,2)  {$s_3\la$};
        		\node (31) at (0,4) {$s_3s_1\la$};
        		\node (21) at (-4,4)  {$s_2s_1\la$};
        		\node (321) at (-2,6) {$s_3s_2s_1\la$};
        		
        		\draw (la) -- node[left] {$1$} (1);
        		\draw (la) -- node[left] {$3$} (3);
        		\draw (1) -- node[left] {$3$} (31);
        		\draw (3) -- node[left] {$1$} (31);
        		\draw (1) -- node[left] {$2$} (21);
        		\draw (21) -- node[left] {$3$} (321);
        		\draw [thick, dotted] (31) -- (321);
        	\end{tikzpicture}
    	};
		\node [below of=Z, yshift=-17mm] {$Z^\circ_{u,+}$};
    	\node (Z') at (6,0) {
        	\begin{tikzpicture}[scale=0.7]
        		\node (la) at (0,0) {$\emptyset$};
        		\node (1) at (-2,2) {$\{1\}$};
        		\node (3) at (2,2)  {$\{3\}$};
        		\node (31) at (0,4) {$\{1,3\}$};
        		\node (21) at (-4,4)  {$\{1,2\}$};
        		\node (321) at (-2,6) {$\{1,2,3\}$};
        		
        		\draw (la) -- node[left] {$1$} (1);
        		\draw (la) -- node[left] {$3$} (3);
        		\draw (1) -- node[left] {$3$} (31);
        		\draw (3) -- node[left] {$1$} (31);
        		\draw (1) -- node[left] {$2$} (21);
        		\draw (21) -- node[left] {$3$} (321);
        		\draw [thick, dotted] (31) -- (321);
        	\end{tikzpicture}
    	};
		\node [below of=Z', yshift=-17mm] {$Z^{'\circ}_{u,+}$};
	\end{tikzpicture}
	\caption{
		The posets $Z_{u,+}^\circ$ ($\simeq Z_{u w_0^J,+}$) and 
		$Z^{'\circ}_{u,+}$ ($= Z'_{u w_0^J,+}$) 
		for $u=w_{\la}$ where $k=3$ and $\la=(3,2,1)\in\Pk[3]$ 
		(and $w_0^J$ is the longest element of $S_4$).
		Left weak covers are represented as solid lines,
		and strong covers are solid or dotted lines.
		A solid edge between $v$ and $w$ is labelled with $i$ if
		$v=s_i w$.
	}
	\label{fig:ZandZ'}
\end{figure}
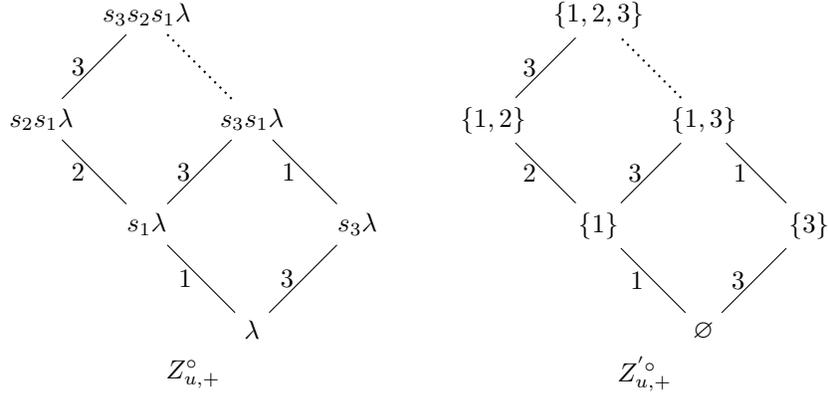

From the example above,
we would expect these properties:
\begin{enumerate}[label=\textup{(\arabic*)}]
	\item
		$Z'_{u,\pm}$ is closed under intersection and union.
	\item
		$Z'_{u,\pm}$ has the maximum element. % of size $k=|I|-1$.
	\item
		$Z_{u,\pm}$ and $Z'_{u,\pm}$ have the Chain Property. 
		(See Section \ref{sect:PropWS::CP} for the detail)
\end{enumerate}
(1), (2), (3) are proved in 
Section
\ref{sect:PropWS::Structure},
\ref{sect:PropWS::Forbidden},
\ref{sect:PropWS::CP},
respectively.

\subsection{Intersection and union} \label{sect:PropWS::Structure}

In this section we prove the following proposition
as the compilation of Lemma \ref{theo:dA_meet}, \ref{theo:dA_join} and \ref{theo:dA_cup}.
\begin{prop}\label{theo:Zu_props}
	For $u\in W$, we have
	\begin{enumerate}[label=\textup{(\arabic*)}]
		\item
			$A,B\in Z'_{u,\pm}$ and $A\cup B\neq I$ $\implies A\cup B\in Z'_{u,\pm}$.
		\item
			$A,B\in Z'_{u,\pm} \implies A\cap B\in Z'_{u,\pm}$.
		\item
			$A,B\in Z'_{u,+} \implies d_{A\cap B}u=(d_Au)\wedge(d_Bu)$.
		\item
			$A,B\in Z'_{u,-} \implies d_{A\cap B}^{-1}u=(d_A^{-1}u)\vee(d_B^{-1}u)$.
	\end{enumerate}
\end{prop}

\vspace{2mm}
In this section we say $A,B\subset I$ are \textit{strongly disjoint}
if for any $i\in A$ and $j\in B$ it holds that $i-j\not\equiv 0,\pm 1$,
and $x,y\in W$ are \textit{strongly commutative}
if any Coxeter generator $s_i$ appearing in a reduced expression of $x$ and
any $s_j$ appearing in that of $y$ satisfy $i-j\not\equiv 0,\pm 1$.
The next lemma is straightforward.
\begin{lemm}\label{theo:strongly}
	Let $A,B\subsetneq I$ and $x,y\in W$.
    \begin{enumerate}[label=\textup{(\arabic*)}]
    	\item
			If $A,B$ are strongly disjoint,
			then $d_A, d_B$ are strongly commutative.
    	\item
    		For the decomposition $A=A_1\sqcup\dots\sqcup A_m$ into connected components,
			$A_1,\dots,A_m$ are pairwise strongly disjoint and
    		$d_{A_1},\dots,d_{A_m}$ are pairwise strongly commutative.
    	\item
    		For $x'\le x$ and $y'\le y$,
			if $x,y$ are strongly commutative then so are $x',y'$.
    	\item
    		If $x,y$ are strongly commutative, then $x,y$ are commutative and $l(xy)=l(x)+l(y)$.
    \end{enumerate}
\end{lemm}

\begin{lemm}\label{theo:sc_xyz}
	Let $x,y,z\in W$ with $x,y$ are strongly commutative.
	\Todo{Is strong commutativity necessary?}
	Then
	\begin{enumerate}[label=\textup{(\arabic*)}]
		\item
			$z\le_L xyz \iff z\le_L xz,yz$.
		\item
			$z\ge_L xyz \iff z\ge_L xz,yz$.
	\end{enumerate}
\end{lemm}
\begin{proof}
	\noindent (1)
	The ``only if'' direction immediately follows by the definition of the weak order and commutativity of $x,y$.
	We prove the ``if'' direction by induction on $l(x)+l(y)$.
	It is clear when $l(x)=0$ or $l(y)=0$.
	In particular the case $l(x)+l(y)\le 1$ is done 
	and we may assume $l(x)+l(y)\ge 2$ and $l(x),l(y)>0$.
	
	\noindent\underline{\textit{Step A}}: the case $l(x)+l(y)=2$, i.e. $l(x)=l(y)=1$.
	
	We can write $x=s_i$ and $y=s_j$ with $s_i\neq s_j$, $s_i s_j=s_j s_i$ from the strong commutativity.
	We have $s_i z, s_j z \ge_L z$ by the assumption.
	Hence $z\in W/W_{\{i,j\}}$, where $W_{\{i,j\}}=\langle s_i,s_j\rangle=\{e,s_i,s_j,s_i s_j\}$.
	Therefore $s_i s_j z\ge_L z$.
	
	\noindent\underline{\textit{Step B:}} the case $l(x)+l(y)>2$.
	
	From the commutativity of $x,y$ we may assume $l(y)\ge l(x)$; in particular $l(y)>1$.
	Take a reduced expression of $y=s_{i_1}\dots s_{i_l}$ and put $y'=s_{i_1}\dots s_{i_{l-1}}$,
	$z'=s_{i_l}z$.
	Since $z\le_L yz$ and $s_{i_l}\le_L y$,
	we have $z\le_L z'$.
	Now we can obtain $z'\le_L xy'z'$,
	which implies $z\le_L z'\le_L xy'z'=xyz$ as desired,
	by applying the induction hypothesis
	for $(x,y,z) := (x,y',z')$,
	having its assumption satisfied as follows:
	\begin{itemize}
		\item
			$x,y'$ are strongly commutative. \\
			\textit{Proof.}
			From Lemma \ref{theo:strongly} (3).
		\item
			$z'\le_L y'z'$. \\
			\textit{Proof.}
			Since $z\le_L yz$ and $s_{i_l}\le_L y$, by Lemma \ref{theo:weakorder_xyz}(1)
			we have $z' = s_{i_l}z \le_L yz = y'z'$.
		\item
			$z'\le_L xz'$. \\
			\textit{Proof.}
			Since $l(x)+l(y)>l(x)+l(s_{i_l})$,
			we can obtain $z\le_L xz'$
			by applying the induction hypothesis for $(x,y,z):=(x,s_{i_l},z)$,
			having that its assumption described below is clearly satisfied:
			\begin{itemize}
				\item
					$x$ and $s_{i_l}$ are strongly commutative.
				\item
					$z\le_L xz$.
				\item
					$z\le_L s_{i_l}z$.
			\end{itemize}
			Besides $s_{i_l}\le_L xs_{i_l}$,
			hence 
			we have $z'\le_L xz'$
			by Lemma \ref{theo:weakorder_xyz}(1).
	\end{itemize}

	\noindent (2) is proved similarly to (1).
\end{proof}

\begin{lemm}\label{theo:dA_meet}
	Let $w\in W$ and $A, B\subsetneq I$ with
	$w\le_L d_Aw, d_Bw$.
	\begin{enumerate}[label=\textup{(\arabic*)}]
		\item
			$w\le_L d_{A\cap B}w$.
		\item
			The element $d_{A\cap B}w$ is the strong meet of $d_{A}w$ and $d_{B}w$.
	\end{enumerate}
\end{lemm}

\begin{rema}
The same statement with all $d_X$ replaced with $u_X$ is proved similarly.
\end{rema}

\begin{rema}
	It does not generally hold that
	if $w\le_L xw,yw$ and $x\wedge y$ exists then $w\le_L (x\wedge y)w$;
	a counterexample is $W=S_4$, $x=s_{21}$, $y=s_{23}$, $w=s_{2}$.
\end{rema}

\begin{proof}
	\noindent (1)
	Within this proof we call $x\in W$ satisfies $(*)$ if $w\le_L xw$.
	
	Decomposing $A, B$ into connected components
	$A = A_1 \sqcup \dots \sqcup A_m$ and
	$B = B_1 \sqcup \dots \sqcup B_n$,
	we have
	$A\cap B = \bigsqcup_{i,j} (A_i \cap B_j)$.
	Each nonempty $A_i \cap B_j$ has at most two connected components,
	each component $C$ of which satisfies
	$d_{A_i} = x d_{C}$ for some $x\in W$ or 
	$d_{B_j} = y d_{C}$ for some $y\in W$ as easily seen.
	Having that both $d_{A}$ ($\ge_L d_{A_i}$) and 
	$d_{B}$ ($\ge_L d_{B_j}$) satisfy $(*)$
	and that lower bounds in $\le_{L}$ inherit $(*)$,
	we see each $d_C$ satisfies $(*)$.
	Besides $(A_i \cap B_j) \cap (A_{i'} \cap B_{j'}) = (A_i \cap A_{i'}) \cap (B_{j} \cap B_{j'})$
	is empty unless $(i,j)=(i',j')$,
	we thus have $A \cap B$ decomposes as $A\cap B = C_1 \sqcup \dots \sqcup C_l$ into connected components,
	where each $d_{C_i}$ satisfies $(*)$.
	Now it follows from Lemma \ref{theo:sc_xyz} (1) that
	$d_{A\cap B} = d_{C_1} \dots d_{C_l}$ satisfies $(*)$, as desired.
	
	\noindent (2)
	By the Subword Property we have $d_{A\cap B}=d_A\wedge d_B$.
	From the assumption and (1),
	we have $\phi^R_{w}(d_X)=d_Xw$ for $X=A,B,A\cap B$.
	Hence by Lemma \ref{theo:b_i} (5) we have $d_{A\cap B}w=d_Aw\wedge d_Bw$.
\end{proof}

\begin{coro}\label{theo:meetSemiLattice}
	Let $\la\in\Pk$,
	and $\ka^{(1)}, \ka^{(2)}$ be weak strips over $\la$.
	Write $\ka^{(i)} = d_{A_i} \la$ for each $i$ with $A_i\subsetneq I$.
	Then $d_{A_1 \cap  A_2} \la$ is a weak strip over $\la$ and
	is the meet of $\ka^{(1)},\ka^{(2)}$
	in the poset $\Pk$ with the strong order:
	$\displaystyle \ka^{(1)} \wedge \ka^{(2)} = d_{A_1 \cap  A_2} \la$.
\end{coro}

\begin{proof}
	Let $w_{\la}\in \Wo$ be the affine Grassmannian permutation corresponding to $\la$,
	and $w_0$ the longest element of $S_{k+1}$.
	By Lemma \ref{theo:weakorder_xyz},
	the condition $d_A\la/\la$ is a weak strip is equivalent to
	$d_A w_\la w_0 \ge_L w_\la w_0$.
	From this and Lemma \ref{theo:dA_meet}(1)
	we see $d_{A_1\cap A_2}\la/\la$ is a weak strip.
	From Lemma \ref{theo:dA_meet}(2)
	we have 
	$d_{A_1\cap A_2} w_\la = (d_{A_1} w_\la) \wedge (d_{A_2} w_\la)$
	in $W$.
	Since $\Wo\subset W$ is a subposet, this is also the meet in $\Wo\simeq\Pk$.
\end{proof}

\begin{lemm}\label{theo:dA_join}
	Let $w\in W$ and $A, B\subsetneq I$ with
	$d^{-1}_Aw, d^{-1}_Bw \le_L w$.
	\begin{enumerate}[label=\textup{(\arabic*)}]
		\item
			$d^{-1}_{A\cap B}w \le_L w$.
		\item
			The element $d^{-1}_{A\cap B}w$ is the strong join of $d^{-1}_Aw$ and $d^{-1}_Bw$.
	\end{enumerate}
\end{lemm}

\begin{proof}
	\noindent
	(1) is proved parallelly to Lemma \ref{theo:dA_meet} (1),
	making use of Lemma \ref{theo:sc_xyz}(2) instead of Lemma \ref{theo:sc_xyz}(1).
	
	\noindent
	(2)
	We have $d^{-1}_Aw, d^{-1}_Bw, d^{-1}_{A\cap B}w \in [e,w]_L$ by (1).
	The map $\Phi_w$ in Lemma \ref{theo:anti_isom} sends $d^{-1}_Aw, d^{-1}_Bw, d^{-1}_{A\cap B}w$ to $d_A, d_B, d_{A\cap B}$ respectively.
	Since $d_{A\cap B} = d_A \wedge d_B$, sending them back via $\Psi_w$,
	we have $d^{-1}_{A\cap B}w = (d^{-1}_Aw) \vee (d^{-1}_Bw)$ by Lemma \ref{theo:anti_isom}(3).
\end{proof}

\begin{lemm}\label{theo:dA_cup}
	Let $u\in W$ and $A,B\subsetneq I$ with $A\cup B\neq I$.
	\begin{enumerate}[label=\textup{(\arabic*)}]
		\item \label{enum:dA_cup:Z+}
			If $d_A u, d_B u \ge_L u$, then $d_{A\cup B} u\ge_L u$.
		\item \label{enum:dA_cup:Z-}
			If $d^{-1}_A u, d^{-1}_B u \le_L u$, then $d^{-1}_{A\cup B} u\le_L u$.
	\end{enumerate}
\end{lemm}

\begin{proof}
	We only give a proof of \ref{enum:dA_cup:Z+}
	since that of \ref{enum:dA_cup:Z-} is quite similar.

	Assume $d_Au, d_Bu \ge_L u$.
	Take the decomposition
	$A = A_1 \sqcup\dots\sqcup A_m$ and
	$B = B_1 \sqcup\dots\sqcup B_n$
	into connected components.
	Since $d_{A_i}\le_L d_A$, we have $d_{A_i}u\ge_L u$ for any $i$,
	and similarly $d_{B_j}u\ge_L u$ for any $j$.
	Since $A\cup B = (\dots(A\cup B_1)\cup \dots)\cup B_n$,
	we only need to prove it when $B$ is connected.
	Assume $B$ is connected.
	It is also easy to see, from Lemma \ref{theo:strongly} and Lemma \ref{theo:sc_xyz}(1),
	that it suffices to prove it
	when $A$, $B$ and $A\cup B$ are connected.
	We therefore assume $A$, $B$ and $A\cup B$ are connected.
	The case $A\subset B$ or $B\subset A$ being clear,
	we assume $A\not\subset B$ and $B\not\subset A$;
	namely we let $A=[i,j]$ and $B=[p,q]$ with $p\le i\le q+1\le j+1$ without loss of generality,
	where we employ an ordering $r+1<\dots<k<0<\dots<r-1$ of $I\sm\{r\}$ with an arbitrarily fixed element $r\in I\sm(A\cup B)$.
	Since $d_B=s_q\dots s_p \ge_L s_{i-1}\dots s_p=d_{B\sm A}$ and $d_Bu \ge_L u$,
	we have $d_{B\sm A}u\ge_L u$.
	Hence we may replace $B$ by $B\sm A$ ($=[p,i-1]$).
	
	Let $B'=B\sm\{i-1\}=[p,i-2]$ and $u'=d_{B'}u$.
	Since $d_{B'}\le_L d_B$ and $d_B u\ge_L u$, it follows that $u'\ge_L u$.
	Since $s_{i-1}u'=d_Bu\ge_L u$ and $d_Au'=d_Ad_{B'}u\ge_L u$,
	the latter of which is from Lemma \ref{theo:sc_xyz} (1),
	it easily follows that
	$s_{i-1}u'\ge_L u'$ and $d_Au'\ge_L u'$ from Lemma \ref{theo:weakorder_xyz}.
	
	Toward a contradiction,	
	suppose $d_{A\cup B}u\not\ge_L u$.
	Then we have $d_A s_{i-1} u'\not\ge_L u'$
	since $d_{A\cup B}u=d_A s_{i-1}u'$ and $u\le_L u'$.
	Since $s_{i-1}u'\ge_L u'$,
	there exists $a\in[i,j]$ such that
	$x s_{i-1} u' \ge_L u'$ and
	$s_a x s_{i-1} u' \not\ge_L u'$, which implies $s_a x s_{i-1} u' \lecov x s_{i-1} u'$,
	where we write $x=s_{a-1} s_{a-2}\dots s_{i+1} s_{i}$.
	On the other hand, since $d_Au'\ge_L u'$ we have
	$s_a xu' \ge_L u'$ and
	$x u' \ge_L u'$.
	Besides we have $x s_{i-1}u' \gecov xu'$ from the Subword Property.
	Hence the Lifting Property implies that $xu'\le s_a x s_{i-1} u'$,
	which is actually an equality since both sides have the same length.
	Therefore we have ($s_{a-1} s_{a-2} \dots s_{i+1}s_{i} =$) $x = s_a x s_{i-1}$ ($=s_a s_{a-1} \dots s_i s_{i-1}$),
	which is absurd.	
\end{proof}

\begin{rema}
    Unlike the ``cap'' case, it does not always hold that 
    $d_{A\cup B}u = (d_Au)\vee(d_Bu)$ in (1),
    or $d^{-1}_{A\cup B}u=(d^{-1}_Au)\wedge(d^{-1}_Bu)$ in (2).
    A counterexample for (1) is given by $W=S_3$, $u=e$, $A=\{1\}$ and $B=\{2\}$.
\end{rema}

\subsection{Non-appearing indices}\label{sect:PropWS::Forbidden}

\begin{prop}\label{theo:forbiddenindex}
	\noindent $(1)$
	For any $w\in W$,
	there exists $i_w^-\in I$ such that 
	$i_w^-\notin A$ for any $A\subsetneq I$ with
	$d_A^{-1}u\le_L u$.
	
	\noindent $(2)$
	For any $w\in \Wo$,
	there exists $i_w^+\in I$ such that 
	$i_w^+\notin A$ for any $A\subsetneq I$ with
	$d_Au\ge_L u$ and $d_Au\in\Wo$.
\end{prop}
\begin{proof}
	\noindent
	(1)
	For any $A\subsetneq I$, we have
	\begin{align*}
		d_A^{-1}w \le_L w
			&\iff d_A\le_R w \\
			&\iff u_A\le_L w^{-1} \\
			&\implies \mathrm{RI}(u_A)\subset \mathrm{RI}(w^{-1}),
	\end{align*}
	and the last condition is equivalent to
	$A$ being included by the first row of $\mathrm{RI}(w^{-1})$.
	Hence we can take $i_w^-$ from the complement of the first row of $\mathrm{RI}(w^{-1})$.
	
	\noindent
	(2)
	By Lemma \ref{theo:join_exist} we may take
	$z := \bigvee_L\{d_Aw\mid A\subsetneq I \text{ s.t. } d_Aw\ge_L w,\,d_Aw\in\Wo\}$,
	the left join of all weak strips over $w$.
	Take any $A\subsetneq I$ such that $d_Aw\ge_L w$ and $d_Aw\in\Wo$.
	Since $w,d_A w\le_L z$, we have
	$zw^{-1}\ge_R z(d_Aw)^{-1} = zw^{-1}u_A$, which is equivalent to
	$wz^{-1}\ge_L d_Awz^{-1}$.
	Hence, similarly to the proof of (1) we have
	$A$ is a subset of the first row of
	$\mathrm{RD}((wz^{-1})^{-1})=\mathrm{RD}(zw^{-1})$,
	which is a proper subset of $I$ and independent of $A$,
	and therefore we can take $i_w^+$ from its complement.
\end{proof}

\begin{rema}
The index $i_w^+$ in (2) above is in fact uniquely determined as follows:
a bounded partition $\la\in\Pk$,
corresponding to a 0-dominant affine permutation $w_\la\in\Wo$,
has the unique weak strip of size $k$, namely $(k)\cup\la$.
Since the corresponding core $\core((k)\cup\la)$ has $k$ more boxes in the first row
than $\core(\la)$ does,
the only possibility for $i_{w_\la}^+$ is
what is determined by the following equivalent descriptions:
\begin{itemize}
	\item
		The residue of the rightmost box in the first row of $\core(\la)$.
	\item
		The negative of the residue written in the leftmost box in the last row of
		$\mathrm{RI}(w_\la) = \la^{\omega_k}$.
	\item
		$m-1$,
		where $w_\la=u_{A_m}\dots u_{A_1}$
		is the maximal increasing decomposition for $w_\la$.
		(Note that $A_m=\{i,i+1,\dots,m-2,m-1\}$ for some $i$.)
\end{itemize}
\end{rema}
\begin{figure}
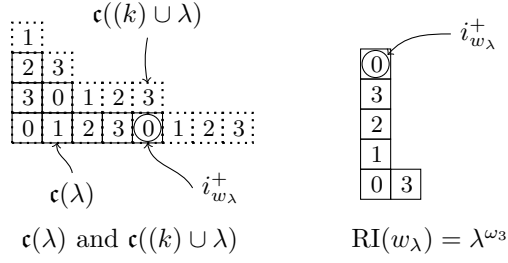

	\centering
	\tikzitem{
		\node (core) at (0,0) {
			\tikzitem[0.4]{
				\tikzydiag[thick, dotted]{8,5,2,1};
				\tikzydiag{5,2,1};
				\tikzvnum{1}{0,1,2,3,0,1,2,3};
				\tikzvnum{2}{3,0,1,2,3};
				\tikzvnum{3}{2,3};
				\tikzvnum{4}{1};
				
				\draw [<-] (1.5,0) to [out=-100,in=100] (2,-1) node[below]{$\core(\la)$};
				\draw [<-] (4.5,2) to [out=100,in=-100] (4.5,3.5) 
					node[above]{$\core((k)\cup\la)$};
				
				\draw (4.5,0.5) circle (0.45);
				\draw [<-] (4.5, 0) to [out=-80, in = 120] (6, -1.5)
					node[right]{$i_{w_\la}^+$};
			}
		};
		\node [below of=core, yshift=-8mm] {
			$\core(\la)$ and $\core((k)\cup\la)$
		};
		\node (RI) at (4,0) {
			\tikzitem[0.4]{
				\tikzydiag{2,1,1,1,1};
				\tikzvnum{1}{0,3};
				\tikzvnum{2}{1};
				\tikzvnum{3}{2};
				\tikzvnum{4}{3};
				\tikzvnum{5}{0};
								
				\draw (0.5,4.5) circle (0.45);
				\draw [<-] (0.9, 4.8) to [out=50, in = 180] (3, 5.5)
					node[right]{$i_{w_\la}^+$};
			}
		};
		\node [below of=RI, yshift=-8mm] {
			$\RI(w_\la) = \la^{\omega_3}$
		};
	}
	\label{fig:forbidden_index}
	\caption{
		An example where $k=3$, $\la=(3,2,1)$ and $\core(\la)=(5,2,1)$.
		The dotted shape on the left figure represents $\core((k)\cup\la)$,
		and the solid one does $\core(\la)$.
		In this case $w_{(k)\cup\la} = s_3 s_2 s_1 w_\la = d_{\{1,2,3\}}w_\la$
		and therefore
		$i_{w_\la}^+=0$.
	}
\end{figure}

\begin{rema}
We cannot drop the assumption on $0$-dominantness of $d_Aw$ in (2) of the proposition.
For example, let $k=3$ and $w=s_3 s_0$.
Then $w=u_{\{3,0\}}$ is the maximal increasing decomposition
and hence $i_w^+$ should be $0$,
but $d_{\{0\}}w=s_0 s_3 s_0 \ge_L w$.
\end{rema}

\begin{coro}\label{theo:Zu_max}
	Let $u\in W$.
	\begin{enumerate}[label=\textup{(\arabic*)}]
		\item
			$Z'_{u,+}$ has the maximum element under $\subset$.
			Hence, $Z_{u,+}$ has the maximum element under $\le$.
		\item
			$Z'_{u,-}$ has the maximum element under $\subset$.
			Hence, $Z_{u,-}$ has the minimum element under $\le$.
	\end{enumerate}
\end{coro}
\begin{proof}
	By Proposition \ref{theo:Zu_props} (1) and Proposition \ref{theo:forbiddenindex}.
\end{proof}

\subsection{Chain Property}\label{sect:PropWS::CP}

\begin{prop}\label{theo:CP_Z+-}
	The sets $Z_{u,+}$ and $Z_{u,-}$ have the Chain Property.
	Namely, for any $x,y\in Z_{u,\pm}$ such that $x\le y$,
	there exists a sequence
	$x=\exists z^{(0)}\lecov \exists z^{(1)} \lecov\dots\lecov \exists z^{(l)}=y$
	such that $z^{(i)}\in Z_{u,\pm}$ for any $i$.
	\Todo{Can it be simplified with projection operators?}
\end{prop}
\begin{proof}
	First we note a few immediate observations:
	\begin{itemize}
		\item
			For a poset $P$ and a subposet $Q\subset P$,
			if $A\subset P$ is an order ideal
			then $A\cap Q$ is an order ideal of $Q$.
		\item
			If a subset $X$ of a Coxeter group $W$ has the Chain Property and 
			$Y\subset X$ is an order ideal, then $Y$ also has the Chain Property.
	\end{itemize}
	
	Let $D = \{d_A\mid A\subsetneq I\}$.
	Since $D\subset W$ is an order ideal,
	the set $\{d_A\mid d_A\le_R u\} = D \cap [e,u]_R$ is an order ideal of $[e,u]_R$
	and hence has the Chain Property since
	so does $[e,u]_R$ as proved in Theorem \ref{theo:CP_weakorderideal}.
	Hence
	$Z_{u,-}$ also has the Chain Property since
	it is the image of $\{d_A\mid d_A\le_R u\}$
	under the the anti-isomorphism $\Psi_u\colon[e,u]_R\lra[e,u]_L ; x\mapsto x^{-1}u$.

	Similarly,
	the set $\{d_A\mid d_Au\ge_L u\} = D \cap (W/\{u\})$ has the Chain Property
	since
	it is an order ideal of $W/\{u\}$,
	which has the Chain Property \cite[Corollary 3.5]{MR946427}.
	Hence, 
	since $Z_{u,+}$ is the image of $\{d_A\mid d_Au\ge_L u\}$
	under the isomorphism $(\cdot u)\colon W/\{u\}\lra[u,\infty)_L$,
	we conclude that $Z_{u,+}$ has the Chain Property.
\end{proof}

From the isomorphism
$(Z_{u,+},\le) \simeq (Z'_{u,+},\subset)$
and the anti-isomorphism
$(Z_{u,-},\le) \underset{\text{anti}}{\simeq} (Z'_{u,-},\subset)$,
we have the Chain Property for $Z'_{u,\pm}$:

\begin{coro}\label{theo:CP_Z'+-}
	The sets $Z'_{u,+}$ and $Z'_{u,-}$ have the Chain Property.
	Namely, for any $A,B\in Z'_{u,\pm}$ with $A\subset B$,
	there exists a sequence
	$A=\exists C^{(0)}\subsetcov \exists C^{(1)} \subsetcov\dots\subsetcov \exists C^{(l)}=B$
	such that $C^{(i)}\in Z'_{u,\pm}$ for any $i$.
\end{coro}

\section{Proof of the Pieri rule for $\kkss{\la}$}\label{sect:StrongSumPieri}

This section is devoted for the proof of
Theorem \ref{theo:StrongSumPieri} and \ref{theo:StrongSumPieri'}.

\subsection{Outline}\label{sect:StrongSumPieri::Outline}

	Let $w = w_{\la} \in \Wo$ be the affine Grassmannian element corresponding to $\la$.
	Recall the Pieri rule for $\kks{\la}$ (Definition \ref{defi:KkPieri}):
	$$
		\kks{v} h_i = 
			\sum_{\substack{
				A \subset I, |A| = i \\
				d_A*v \in \Wo
			}}
			(-1)^{i-(l(d_A*v)-l(v))}
			\kks{d_A*v}.
	$$
	Summing this up over
	$v\in \Wo\cap[e,w]$ and
	$i \in\{0,1,\dots,r\}$,
	we have
	$$
		\kkss{w} \wth_r = 
			\sum_{\substack{v\le w \\ v\in \Wo}}
			\sum_{\substack{
				A \subset I, |A| \le r \\
				d_A*v \in \Wo
			}}
			(-1)^{|A|-(l(d_A*v)-l(v))}
			\kks{d_A*v},
	$$
	and its coefficient of $\kks{u}$ (for $u\in \Wo$) is
	\begin{equation}\label{eq:gu_coeff}
		[\kks{u}](\kkss{w} \wth_r) =
			\sum_{\substack{
				v \le w \\
				v \in \Wo
			}}
			\sum_{\substack{
				A \subset I, |A| \le r \\
				u = d_A*v
			}}
			(-1)^{|A|-(l(u)-l(v))}.
	\end{equation}
	
	We shall illustrate,
	in the example below,
	that
	if the summation above is not empty then
	there is a ``matching'' on the set of appearing $(A,v)$'s
	with an unmatched element,
	and the corresponding summands cancel accordingly,
	and consequently the value of the summation is equal to 1.
	
\begin{exam}
	Let $k=3$ and $u=s_{310}=w_{\la} \in \aG[4]$ where $\la=(2,1)\in\Pk[3]$.
	Table \ref{table:(v,A)} lists the pairs $(v,A)$ such that
	$d_A*v=u$, organized according to the size of $A$.
	Apparently there are the same number of 
	pairs $(v,A)$ with $|A|=r'$ and $(-1)^{|A|-(l(u)-l(v))}=+1$,
	and those with $|A|=r'$ and $(-1)^{|A|-(l(u)-l(v))} = -1$,
	for each fixed $r'>0$.
	Furthermore, introducing the condition $v\le w$ for $w=s_{210}$, say,
	we see that
	the set of the pairs $(v,A)$ with 
	$d_A*v=u$ and $v\le w$ is
	$\{(s_{10},\{3\}), (s_{0},\{1,3\}), (s_{10},\{1,3\})\}$,
	and that
	the number of such pairs $(v,A)$ with $|A|=r'$ and $(-1)^{|A|-(l(u)-l(v))}=+1$
	and those with $|A|=r'$ and $(-1)^{|A|-(l(u)-l(v))}=-1$
	coincide whenever $r'\neq 1$, and differ by $1$ when $r'=1$.
\end{exam}
%	\vspace{2mm}
	
\begin{figure}
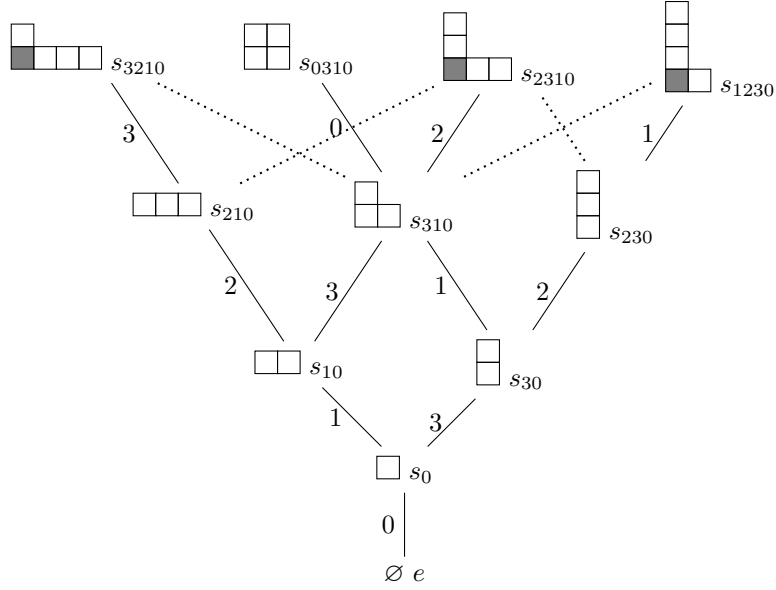

	\centering
	\tikzitem[1.4]{
		\node (0) at (0,0) {$\emptyset$ $e$};
    	\node (1) at (0,1) {$\sk[0.3]{1}{}$ $s_0$};
    	\node (2) at (-1,2) {$\sk[0.3]{2}{}$ $s_{10}$};
    	\node (11) at (1,2) {$\sk[0.3]{1,1}{}$ $s_{30}$};
    	\node (3) at (-2,3.5) {$\sk[0.3]{3}{}$ $s_{210}$};
    	\node (21) at (0,3.5) {$\sk[0.3]{2,1}{}$ $s_{310}$};
    	\node (111) at (2,3.5) {$\sk[0.3]{1,1,1}{}$ $s_{230}$};
    	\node (31) at (-3,5) {$\sk[0.3]{4,1}{1}$ $s_{3210}$};
    	\node (22) at (-1,5) {$\sk[0.3]{2,2}{}$ $s_{0310}$};
    	\node (211) at (1,5) {$\sk[0.3]{3,1,1}{1}$ $s_{2310}$};
    	\node (1111) at (3,5) {$\sk[0.3]{2,1,1,1}{1}$ $s_{1230}$};
    	\draw 
			(0) -- node[left]{$0$} (1)
    		(1) -- node[left]{$1$} (2)
    		(1) -- node[left]{$3$} (11)
    		(2) -- node[left]{$2$} (3)
    		(2) -- node[left]{$3$} (21)
    		(11) -- node[left]{$1$} (21)
    		(11) -- node[left]{$2$} (111)
    		(3) -- node[left]{$3$} (31)
    		(21) -- node[left]{$0$} (22)
    		(21) -- node[left]{$2$} (211)
    		(111) -- node[left]{$1$} (1111);
    	\draw[thick, dotted]
    		(3) -- (211)
    		(21) -- (31)
    		(21) -- (1111)
    		(111) -- (211);

	}
	\label{fig:weakposet_k=3}
	\caption{
		The poset of 4-cores (and corresponding elements in $\tilde{S_4^\circ}$) 
		up to those of size $4$.
		The left weak covers are represented by solid lines,
		and the strong covers are dotted or solid lines.
		A solid edge labelled with $i$ corresponds to the left multiplication by $s_i$.
	}
\end{figure}
\begin{table}[h]
    \caption{
    	The list of $(v,A)$ such that $d_A*v=u$, where $u=s_{310}$.
    }
    \label{table:(v,A)}
    \centering
    \begin{tabular}{ccc}
    	& $(v,A)$ & $(-1)^{|A|-(l(u)-l(v))}$ \\ \hline
    	$|A|=0$ & $(s_{310},\emptyset)$ & $+1$ \\ \hline
    	$|A|=1$ & $(s_{30},\{1\})$ & $+1$ \\
    	 & $(s_{310},\{1\})$ & $-1$ \\
    	 & $(s_{10},\{3\})$ & $+1$ \\
    	 & $(s_{310},\{3\})$ & $-1$ \\ \hline
    	$|A|=2$ & $(s_{0},\{1,3\})$ & $+1$ \\
    	 & $(s_{10},\{1,3\})$ & $-1$ \\
    	 & $(s_{30},\{1,3\})$ & $-1$ \\
    	 & $(s_{310},\{1,3\})$ & $+1$ \\ \hline
    \end{tabular}
\end{table}

According to the observation above,
we let
\begin{align*}
	X_{A,u} &= \{v\in W \mid d_A*v = u\} = \{v\in W \mid \phi_{d_A}(v) = u \},\\
	Y_{A,u} &= X_{A,u} \cap [e,w].
\end{align*}
for $u\in \Wo$ and $A\subsetneq I$.
Note that, for any $v\in X_{A,u}$,
Lemma \ref{theo:b_i}(1) implies $v\le_L u$,
and hence it follows from $u\in\Wo$ that $v\in \Wo$.
Hence
\begin{equation}\label{eq:gu_coeff_sumsplit}
	[\kks{u}](\kkss{w} \wth_r)
		= \sum_{|A|\le r} \sum_{v\in Y_{A,u}} (-1)^{|A|-(l(u)-l(v))}.
\end{equation}

The flow of the proof is as follows:
\begin{enumerate}[label=Step \arabic*.]
	\item %[\textit{Step 1.}]
		Every element of $X_{A,u}$ has the form $d_B^{-1}u$ with $B\subset A$,
		and thereby $X_{A,u}$ is anti-isomorphic to a subposet of $[\emptyset, A]$,
		denoted later by $X'_{A,u}$,
		by $d_B^{-1}u\mapsto B$.
	\item %[\textit{Step 2.}]
		The poset $X'_{A,u}\subset [\emptyset, A]$ has the minimum element $B$ and is a boolean poset; $X'_{A,u}=[B,A]$.
	\item %[\textit{Step 3.}]
		The subset $Y_{A,u}$ of $X_{A,u}$ being an order ideal,
		its image $Y'_{A,u}$ under $X_{A,u}\simeq X'_{A,u}$ is an order filter of $X'_{A,u}$.
		Moreover $Y'_{A,u}$ is closed under intersection, reflecting join-closedness of $Y_{A,u}$.
		Hence $Y'_{A,u}$ is also a boolean lattice.
		Therefore,
		the value of the summation over $v\in Y_{A,u}$ in (\ref{eq:gu_coeff_sumsplit}) 
		is 0 unless $|Y_{A,u}|=1$
		since its summands cancel out,
		and $1$ if $|Y_{A,u}|=1$.
	\item %[\textit{Step 4.}]
		If $u\le d_{B}w$ for some $B\subsetneq I$ with $|B|= r$ and $d_Bw\ge_L w$,
		then there uniquely exists $A$ such that $|Y_{A,u}|=1$,
		and hence the value of the right-hand side in (\ref{eq:gu_coeff_sumsplit}) is $1$.
		If there does not exist such $B$, then neither does such $A$,
		and hence (\ref{eq:gu_coeff_sumsplit}) is $0$.
\end{enumerate}

\begin{rema}
	The set $X_{A,u}$ is a fiber of the Demazure action $\phi_{d_A}$.
	In Step 2 (Corollary \ref{theo:YA_boolean}) this fiber is shown to be a boolean poset.
	Meanwhile, for the longest element $w_J$ of a finite parabolic subgroup $W_J$,
	any fiber of its Demazure action $\phi_{w_J}$ is a parabolic coset $W_Jx$, whence isomorphic to $W_J$.
	More generally it might be interesting to find fibers of the Demazure action $\phi_w$ of an arbitrary element $w$.
\end{rema}

\subsection{Proof of Theorem \ref{theo:StrongSumPieri} and \ref{theo:StrongSumPieri'}}

We fix $u\in \Wo$.

\subsubsection{Step 1}
	
	We fix $A\subsetneq I$.
	Since $Y_{A,u}\subset X_{A,u}$,
	the summation over $v$ in (\ref{eq:gu_coeff_sumsplit}) is 0 when $X_{A,u}=\emptyset$.
	We hence assume $X_{A,u}\neq\emptyset$,
	since otherwise such $A$ does not contribute to the value of the right-hand side of (\ref{eq:gu_coeff_sumsplit}).
	Take arbitrary $v\in X_{A,u}$.
	From Lemma \ref{theo:b_xyz} and the definition of $X_{A,u}$ we have
	\begin{enumerate}[label=\textup{(\arabic*)}]
		\item
			$v, d_A^{-1}u \le_L u$,
		\item
			$d_A^{-1}u\le v$.
	\end{enumerate}
	From Proposition \ref{theo:anti_isom} (1) and (1) above,
	(2) is equivalent to 
	\begin{enumerate}[label=\textup{(\arabic*)}]
		\setcounter{enumi}{2}
		\item
			$uv^{-1}\le d_A$.
	\end{enumerate}
	The Subword Property and (3) imply $uv^{-1}=d_B$, or equivalently $v=d_B^{-1}u$,
	for some $B\subset A$.
	We have $A,B\in Z'_{u,-}$ from (1).

	We let
	\begin{align*}
		X'_{A,u} &= \{ B\subsetneq I \mid d_B^{-1}u \in X_{A,u}\}, \\
		Y'_{A,u} &= \{ B\subsetneq I \mid d_B^{-1}u \in Y_{A,u}\}.
	\end{align*}
	The argument above is restated as follows (see also Figure \ref{fig:XYZ_comm_diag}):
\begin{lemm}\label{theo:XA_XA'}
	\begin{enumerate}[label=\textup{(\arabic*)}]
		\item
			$X_{A,u}\neq\emptyset \implies A\in Z'_{u,-}$.
		\item
			$X_{A,u}\subset [d_A^{-1}u, u]$.
		\item
			$(X'_{A,u}, \subset)$ and $(X_{A,u},\le)$ are anti-isomorphic by $B\mapsto d_B^{-1}u$.
		\item
			$X'_{A,u}\subset [\emptyset,A]$.
		\item
			$X'_{A,u} \subset Z'_{u,-}$.
	\end{enumerate}
\end{lemm}
\begin{proof}
	It remains to show that the mapping $B\mapsto d_B^{-1}u$ in (3) is order-reversing,
	which follows from Proposition \ref{theo:anti_isom} (1) and the Subword Property.
\end{proof}

\begin{figure}
	\caption{Relation between $Z_{u,-},Z'_{u,-},X_{A,u},X'_{A,u},Y_{A,u},Y'_{A,u}$.}
	\label{fig:XYZ_comm_diag}
    \[
    \begin{array}{ccccccc}
    	 & & B & \longmapsto & d_B^{-1}u & & \\
    	 & & \rotatebox{-90}{$\in$} & & \rotatebox{-90}{$\in$} & & \\
    	\text{$[\emptyset,I)$} & \supset & Z'_{u,-} & \underset{\text{anti}}{\simeq} & Z_{u,-} & \subset & [e,u]_L \\
    	\rotatebox{90}{$\subset$} & & \rotatebox{90}{$\subset$} & & \rotatebox{90}{$\subset$} & & \rotatebox{90}{$\subset$} \\
    	 \text{$[\emptyset, A]$} & \supset & X'_{A,u} & \underset{\text{anti}}{\simeq} & X_{A,u} & \subset & [e,u]_L \cap [d_A^{-1}u,u] \\
    	 & & \rotatebox{90}{$\subset$} & & \rotatebox{90}{$\subset$} & & \\
    	 & & Y'_{A,u} & \underset{\text{anti}}{\simeq} & Y_{A,u} & = & X_{A,u}\cap [e,w]
    \end{array}
    \]
\end{figure}

\subsubsection{Step 2 and 3}\label{sect:StrongSumPieri::Proof::Step23}

Let us start with an example to describe the situation.

\begin{exam}\label{exam:XA}
	Let $k=5$, $\la=(5,3,2,1)$, $\mu=(5,2,2,2)$, $u=w_\la$ and $w=w_\mu$
	(see Figure \ref{fig:XA}).
	When $A=\{5,0,1\}$%
	\footnote{In this example we follow the cyclic ordering $3<4<5<0<1$ on $I\sm\{2\}$,
		as we see $i_u^-=2$,
		i.e.\,every element of $Z'_{u,-}$ is a subset of $I\sm\{2\}$.
	},
	for example,
	$X_{A,u}=Y_{A,u}=\{s_{1}u,s_{01}u,s_{51}u,s_{501}u\}$ and
	$X'_{A,u}=Y'_{A,u}=[\{1\},\{5,0,1\}]$.
	Similarly, when $A=\{3,5,1\}$ we see
	$X'_{A,u}=[\emptyset,\{3,5,1\}]$ and
	$Y'_{A,u}=[\{1\},\{3,5,1\}]$.
\end{exam}

\begin{figure}
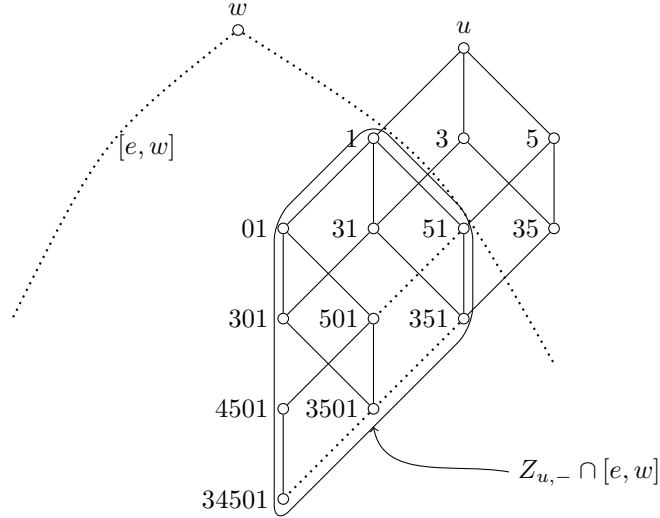

	\centering
	\caption{
		Each vertex labelled with $i_1\dots i_m$ represents
		$s_{i_1}\dots s_{i_m}u \in Z_{u,-}$.
		Left covers are represented by solid edges,
		and strong covers are dotted or solid edges.
	}
	\label{fig:XA}
	\tikzitem[1.2]{
		\tikzstyle{vertex}=
			[draw,circle,fill=white,minimum size=4pt,inner sep=0pt]
		\node [vertex,label=above:$u$] (u) at (0,0) {};
		\node [vertex,label=left:$1$] (1) at (-1,-1) {};
		\node [vertex,label=left:$3$] (3) at (0,-1) {};
		\node [vertex,label=left:$5$] (5) at (1,-1) {};
		\node [vertex,label=left:$01$] (01) at (-2,-2) {};
		\node [vertex,label=left:$31$] (13) at (-1,-2) {};
		\node [vertex,label=left:$51$] (15) at (0,-2) {};
		\node [vertex,label=left:$35$] (35) at (1,-2) {};
		\node [vertex,label=left:$301$] (301) at (-2,-3) {};
		\node [vertex,label=left:$501$] (501) at (-1,-3) {};
		\node [vertex,label=left:$351$] (135) at (0,-3) {};
		\node [vertex,label=left:$4501$] (4501) at (-2,-4) {};
		\node [vertex,label=left:$3501$] (5301) at (-1,-4) {};
		\node [vertex,label=left:$34501$] (34501) at (-2,-5) {};
		
		\draw (u) to (1);
		\draw (u) to (3);
		\draw (u) to (5);
		\draw (1) to (01);
		\draw (1) to (13);
		\draw (1) to (15);
		\draw (3) to (13);
		\draw (3) to (35);
		\draw (5) to (15);
		\draw (5) to (35);
		\draw (01) to (301);
		\draw (01) to (501);
		\draw (13) to (301);
		\draw (13) to (135);
		\draw (15) to (135);
		\draw (35) to (135);
		\draw (301) to (5301);
		\draw (501) to (5301);
		\draw (501) to (4501);
		\draw (4501) to (34501);
%		\draw (u) to node[left]{1} (1);
%		\draw (u) to node[left]{3} (3);
%		\draw (u) to node[left]{5} (5);
%		\draw (1) to node[left]{0} (01);
%		\draw (1) to node[left]{3} (13);
%		\draw (1) to node[right]{5} (15);
%		\draw (3) to node[left]{1} (13);
%		\draw (3) to node[left]{5} (35);
%		\draw (5) to node[below]{1} (15);
%		\draw (5) to node[left]{3} (35);
%		\draw (01) to node[left]{3} (301);
%		\draw (01) to node[left]{5} (501);
%		\draw (13) to node[left]{0} (301);
%		\draw (13) to node[left]{5} (135);
%		\draw (15) to node[left]{3} (135);
%		\draw (35) to node[left]{1} (135);
%		\draw (301) to node[left]{5} (5301);
%		\draw (501) to node[left]{3} (5301);
%		\draw (501) to node[left]{4} (4501);
%		\draw (4501) to node[left]{3} (34501);
		
		\draw [thick,dotted] (15) -- (501);
		\draw [thick,dotted] (135) -- (5301) -- (34501);

		\node [vertex,label=above:$w$] (w) at (-2.5, 0.2) {};
		\draw [thick, dotted] (w) .. controls (-.3,-1.2) .. (1,-3.5);
		\draw [thick, dotted] (w) 
			.. controls (-4,-1) .. 
			node [right] {$[e,w]$}
			(-5,-3);
			
%		\draw [<-] (1)
%			to [out=110,in=-110]
%			(-1,.5) node [above] {$w\SLmeet u = d_{A_0}^{-1}u$};
		
		\draw [rounded corners=8pt]
			(-1,-.8) -- (-2.1,-1.9) -- (-2.1,-5.3) -- (.1,-3.1) -- (.1,-1.9) -- cycle;
		\draw [<-] (-1,-4.2)
			to [out=-80,in=170]
			(.5,-4.7) node [right] {$Z_{u,-}\cap[e,w]$};
	}
\end{figure}
	
\begin{lemm}\label{theo:XA_convex}
	$X_A$ and $Y_A$ are convex under the strong order.
	Namely, if $v\le v' \le v''$ and $v, v''\in X_A$ (resp.\ $Y_A$) then $v'\in X_A$ (resp.\ $Y_A$).
\end{lemm}
\begin{proof}
	It follows from Lemma \ref{theo:b_i}(2).
\end{proof}

\begin{rema}
It is not a very immediate consequence of Lemma \ref{theo:XA_convex} that $X'_{A,u}$ and $Y'_{A,u}$ are convex 
in the boolean poset $[\emptyset, I]$,
yet it is shown to be true in Corollary \ref{theo:YA_boolean}.
\end{rema}
		
In this section we write
$\{i_1,\dots,i_m\}_{<}$ to denote the set $\{i_1,\dots,i_m\}$
for which 
the condition that $(i_1,\dots,i_m)$ is cyclically increasing is imposed.

\begin{lemm}\label{theo:XA_cap}
	\begin{enumerate}[label=\textup{(\arabic*)}]
		\item
			$B,C\in X'_{A,u} \implies B\cap C\in X'_{A,u}$.
		\item
			$B,C\in Y'_{A,u} \implies B\cap C \in Y'_{A,u}$.
	\end{enumerate}
\end{lemm}

\begin{proof}
	(1)
	We prove it by induction on $|A|$.
	The base case $A=\emptyset$ is clear.
	Assume $|A|=m>0$.
	Write $A=\{i_1,\dots,i_m\}_{<}$.
	We need to show
	$\phi_{d_A}(d_{B\cap C}^{-1}u)=u$
	if $\phi_{d_A}(d_{B}^{-1}u)=u$ and $\phi_{d_A}(d_{B}^{-1}u)=u$ for $B,C\subset A$.
	Note that $B\cap C\in Z'_{u,-}$ by Lemma \ref{theo:XA_XA'}(5).
	Let $A'=A\sm\{i_1\}, B'=B\sm\{i_1\}$, $C'=C\sm\{i_1\}$,
	$B''=B\cup\{i_1\}$ and $C''=C\cup\{i_1\}$.
	Note that $\phi_{d_A}=\phi_{i_m}\dots\phi_{i_1} = \phi_{d_{A'}}\phi_{i_1}$.
	
	\vspace{2mm}
	\noindent\underline{\textit{Claim 1.}}
	(a) 
	$\phi_{i_1}(d_B^{-1}u)=d_{B'}^{-1}u$ and
	$\phi_{i_1}(d_C^{-1}u)=d_{C'}^{-1}u$.
	(b) $B'',C''\in Z'_{u,-}$.
	
	\vspace{1mm}
	\noindent\textit{Proof of Claim 1.}
	We only give a proof of the statement for $B$ since that for $C$ is the same.
	
	(Case 1) When $i_1\in B$, we see
	$d_{B''}^{-1}u = d_B^{-1}u = s_{i_1}d_{B'}^{-1}u \lecov d_{B'}^{-1}u$,
	and hence both (a) and (b) is clear.
	
	(Case 2) When $i_1\notin B$,
	we claim that $s_{i_1}d_{B}^{-1}u < d_B^{-1}u$;
	suppose, on the contrary, $s_{i_1}d_{B}^{-1}u > d_B^{-1}u$.
	Then we have $s_{i_1}d_{B}^{-1}u \not\le_L u$ since 
	$l(s_{i_1}d_{B}^{-1}u) > l(u) - l(s_{i_1}d_{B}^{-1})$.
	On the other hand,
	$u=\phi_{d_A}(d_B^{-1}u) = \phi_{d_{A'}}(s_{i_1}d_{B}^{-1}u)$
	since $s_{i_1}d_{B}^{-1}u > d_B^{-1}u$,
	and therefore $s_{i_1}d_{B}^{-1}u\le_L u$ by Lemma \ref{theo:b_xyz},
	which is in contradiction.	
	
	Therefore $s_{i_1}d_{B}^{-1}u < d_B^{-1}u$.
	Now (a) is clear since $d_B^{-1}u = d_{B'}^{-1}u$,
	and (b) follows from $d_{B''}^{-1}u = s_{i_1}d_{B}^{-1}u$.
	\textit{Claim 1 is proved.}
	
	\vspace{2mm}
	\noindent\underline{\textit{Claim 2.}}
	$\phi_{i_1}(d_{B\cap C}^{-1}u)=d_{B'\cap C'}^{-1}u$.
	
	\vspace{1mm}
	\noindent\textit{Proof of Claim 2.}
	By Claim 1(b) and Proposition \ref{theo:Zu_props} (2), we have $B''\cap C'' \in Z'_{u,-}$,
	that is, $u\ge_L d_{B''\cap C''}^{-1}u$.
	Since $B''\cap C'' = (B'\cap C')\cup\{i_1\}$, we have
	$d_{B''\cap C''}^{-1} = s_{i_1} d_{B'\cap C'}^{-1}\gecov d_{B'\cap C'}^{-1}$,
	and hence $d_{B''\cap C''}^{-1}u = s_{i_1} d_{B'\cap C'}^{-1}u \lecov d_{B'\cap C'}^{-1}u$ by Lemma \ref{theo:weakorder_xyz}.
	Noting that $B\cap C = B'\cap C'$ or $B''\cap C''$, 
	in either case $\phi_{i_1}(d_{B\cap C}^{-1}u) = d_{B'\cap C'}^{-1}u$.
	\textit{Claim 2 is proved.}

	\vspace{2mm}
	\noindent\underline{\textit{Claim 3.}}
	$B'\cap C'\in X'_{A',u}$.
	
	\vspace{1mm}
	\noindent\textit{Proof of Claim 3.}
	By Claim 1(a) and that $B\in X'_{A,u}$, we have
	$u = \phi_{d_A}(d_{B}^{-1}u) = \phi_{d_{A'}}\phi_{i_1}(d_{B}^{-1}u) = \phi_{d_{A'}}(d_{B'}^{-1}u)$, 
	and hence $B'\in X'_{A',u}$.
	Similarly $C'\in X'_{A',u}$.
	Hence $B'\cap C'\in X'_{A',u}$ by the induction hypothesis.
	\textit{Claim 3 is proved.}
	
	\vspace{1mm}
	Now we have
	\begin{align*}
		\phi_{d_A}(d_{B\cap C}^{-1}u) 
			&= \phi_{d_{A'}}\phi_{i_1}(d_{B\cap C}^{-1}u) \\
			&= \phi_{d_{A'}}(d_{B'\cap C'}^{-1}u) & &\text{(by Claim 2)} \\
			&= u. & &\text{(by Claim 3)}
	\end{align*}

	\begin{figure}
		\centering
		\label{fig:d_BcapC}
		\caption{For Lemma \ref{theo:XA_cap}}
		\tikzitem{
			\node (case1) at (0,0) {
            	\tikzitem[0.8]{
            		\tikzstyle{vertex}=
            			[draw,circle,fill=white,minimum size=4pt,inner sep=0pt]
            		\draw
            			(0,0) node[vertex] (u) [label=left:$u$]{}
            			-- +(-1,-1) node[vertex] (x1){};
            		\draw [thick, dotted]
            			(x1) -- +(-1.5,-3) node[solid,vertex] (x2) [label=left:$d_{B'}^{-1}u$] {};
            		\draw
            			(x2)
            			to node[right]{$i_1$} 
            			+(-0,-1) node[vertex] (x) [label=left:$d_B^{-1}u$]{};
            			
            		\draw (u)
            			-- +(1,-1) node[vertex] (y1){};
            		\draw [thick, dotted] (y1)
            			-- +(+1.5,-2.5) node[solid,vertex] (y2) [label=left:$d_{C'}^{-1}u$] {};
            		\draw (y2)
            			to node[right]{$i_1$}
            			+(0, -1) node[vertex] (y) [label=left:$d_C^{-1}u$]{};
            
%            		\draw (u)
%            			-- +(0,-1) node[vertex] (z1){};
            		\draw [thick, dotted] (u)
            			-- +(0,-2.5) node[solid,vertex] (z2) [label=left:$d_{B'\cap C'}^{-1}u$] {};
            		\draw (z2)
            			to node[right]{$i_1$}
            			+(0, -1) node[vertex] (z) [label=left:$d_{B\cap C}^{-1}u$]{};
            	}
        	};
			\node [below of=case1, yshift=-15mm] {(When $i_1\in B\cap C$)};
			\node (case2) at (6,0) {
            	\tikzitem[0.8]{
            		\tikzstyle{vertex}=
            			[draw,circle,fill=white,minimum size=4pt,inner sep=0pt]
            		\draw
            			(0,0) node[vertex] (u) [label=left:$u$]{}
            			-- +(-1,-1) node[vertex] (x1) {};
            		\draw [thick, dotted]
            			(x1) -- +(-1.5,-3) 
						node[solid, vertex] (x2) [label=left:$d_B^{-1}u$] {};
            		\draw
            			(x2)
            			to node[right]{$i_1$} 
            			+(-0,-1) node[vertex] (x) [label=left:$d_{B''}^{-1}u$] {};
            			
            		\draw (u)
            			-- +(1,-1) node[vertex] (y1){};
            		\draw [thick, dotted] (y1)
            			-- +(+1.5,-2.5) node[solid, vertex] (y2) [label=left:$d_{C'}^{-1}u$] {};
            		\draw (y2)
            			to node[right]{$i_1$}
            			+(0, -1) node[vertex] (y) [label=left:$d_C^{-1}u$]{};
            
%            		\draw (u)
%            			-- +(0,-1) node[vertex] (z1){};
            		\draw [thick, dotted] (u)
            			-- +(0,-2.5) node[solid, vertex] (z2) [label=left:$d_{B\cap C}^{-1}u$] {};
            		\draw (z2)
            			to node[right]{$i_1$}
            			+(0, -1) node[vertex] (z) [label=left:$d_{B''\cap C''}^{-1}u$] {};
            	}
        	};
			\node [below of=case2, yshift=-15mm] {(When $i_1\notin B$ and $i_1\in C$)};
       	}
	\end{figure}

	(2) follows from (1) and the definition of join and $Y_{A,u}$.
\end{proof}

\begin{lemm}\label{theo:XA_cover}
	Let $A,A'\in Z'_{u,-}$ with $A'\subset A$ and $|A\sm A'|=1$.
	Then $A'\in X'_{A,u}$.
\end{lemm}
\begin{proof}
	Let $A=\{i_1,\dots,i_m\}_{<}$ and $A'=\{i_1,\dots, \widehat{i_k},\dots,i_m\}_{<}$.
	
	Since $u\ge_L d_{A'}^{-1}u = s_{i_1}\dots \widehat{s_{i_k}}\dots s_{i_m} u$,
	\begin{itemize}
		\item
			$\phi_{i_j} (s_{i_j}\dots s_{i_{k-1}} s_{i_{k+1}}\dots s_{i_m} u) = s_{i_{j+1}}\dots s_{i_{k-1}} s_{i_{k+1}}\dots s_{i_m} u$\quad for $1\le j<k$,
		\item
			$\phi_{i_j} (s_{i_j}\dots s_{i_m} u) = s_{i_{j+1}}\dots s_{i_m} u$\quad for $k<j\le m$.
	\end{itemize}

	Since $u\ge_L d_{A}^{-1}u = s_{i_1} \dots s_{i_m} u$,
	\begin{itemize}
		\item
			$\phi_{i_k} (s_{i_{k+1}}\dots s_{i_m} u) = s_{i_{k+1}}\dots s_{i_m} u$.
	\end{itemize}

	Hence
	\begin{align*}
		\phi_{d_A}(d_{A'}^{-1}u)
		&= \phi_{i_m}\dots \phi_{i_{k+1}}\phi_{i_k}\phi_{i_{k-1}}\dots \phi_{i_1} (s_{i_1}\dots s_{i_{k-1}} s_{i_{k+1}}\dots s_{i_m} u) \\
		&= \phi_{i_m}\dots \phi_{i_{k+1}}\phi_{i_k} (s_{i_{k+1}}\dots s_{i_m} u) \\
		&= \phi_{i_m}\dots \phi_{i_{k+1}} (s_{i_{k+1}}\dots s_{i_m} u) \\
		&= u.
	\end{align*}

	\begin{figure}
		\centering
		\label{fig:XA_cover}
		\caption{
			For Lemma \ref{theo:XA_cover}
		}
    	\tikzitem[0.8]{
    		\tikzstyle{vertex}=
    			[draw,circle,fill=white,minimum size=4pt,inner sep=0pt]
    		\draw
    			(0,0) node[vertex] (u) [label=above:$u$]{}
    			to node[left]{$i_m$} 
						+(0,-1) node[vertex] (u1){};
    		\draw [thick, dotted]
    			(u1) -- +(0,-1) node[solid,vertex] (u2) {};
    		\draw
    			(u2)
    			to node[left]{$i_{k+1}$} 
    			+(0,-1) node[vertex] (Y) {};
			\draw (Y)
				to node[left]{$i_{k}$}
				+(-0.5,-1) node[vertex] (l0) {};
			\draw (l0)
				to node[left]{$i_{k-1}$}
				+(0,-1) node[vertex] (l1) {};
			\draw [thick, dotted]
				(l1) -- +(0,-1) node [solid, vertex] (l2) {};
			\draw
				(l2)
				to node[left]{$i_{1}$}
				+(0,-1) node[vertex] (l3) [label=below:$d_A^{-1}u$] {};

			\draw (Y)
				to node[right]{$i_{k-1}$}
				+(0.5,-1) node[vertex] (r1) {};
			\draw [thick, dotted]
				(r1) -- +(0,-1) node [solid, vertex] (r2) {};
			\draw
				(r2)
				to node[right]{$i_{1}$}
				+(0,-1) node[vertex] (r3) [label=below:$d_{A'}^{-1}u$] {};
    	}
	\end{figure}
\end{proof}
	
	\begin{lemm}\label{theo:XA_props}
		Let $A=\{i_1,\dots,i_m\}_{<}\in Z'_{u,-}$ and $B\in X'_{A,u}$.
		By Lemma \ref{theo:XA_XA'}(4) we can write
		$B=\{i_1,\dots,\widehat{i_{j_1}},\dots,\widehat{i_{j_l}},\dots,i_m\}$
		for some $1\le j_1<\dots<j_l \le m$.
		Let $A^{(a)} = \{i_{j_a+1},i_{j_a+2},\dots,i_{m-1},i_m\}$ and
		$B^{(a)}=B\cap A^{(a)}=\{i_{j_a+1},\dots,\widehat{i_{j_{a+1}}},\dots,\widehat{i_{j_{l}}},\dots,i_{m}\}$
		for each $a\in \{1,\dots,l\}$.
		Then,
		for each $1\le a \le l$,
		$$s_{i_{j_a}} d_{B^{(a)}}^{-1} u < d_{B^{(a)}}^{-1} u.$$
	\end{lemm}

	\begin{figure}
		\centering
		\label{fig:XA_props}
		\caption{
			For Lemma \ref{theo:XA_props}
		}
    	\tikzitem[0.8]{
    		\tikzstyle{vertex}=
    			[draw,circle,fill=white,minimum size=4pt,inner sep=0pt]
    		\draw
    			(0,0) node[vertex] (u) [label=above:$u$]{}
    			to node[left]{$i_m$} 
				+(0,-1) node[vertex] (u1){};
    		\draw [thick, dotted]
    			(u1) -- +(0,-1.5) node[solid,vertex] (u2) {};
    		\draw
    			(u2)
    			to node[left]{$i_{j_a+1}$} 
    			+(0,-1) node[vertex] (Ya) [label=right:$d_{B^{(a)}}^{-1}u$] {};
			\draw (Ya)
				to node[left]{$i_{j_a-1}$}
				+(-.5,-1) node[vertex] (l1) {};
			\draw (Ya)
				to node[right]{$i_{j_a}$}
				+(0.3,-1) node[vertex] {};
			\draw [thick, dotted]
				(l1) -- +(0,-1.5) node [solid, vertex] (l2) {};
			\draw
				(l2)
				to node[left]{$i_{j_1+1}$}
				+(0,-1) node[vertex] (Yl) [label=right:$d_{B^{(1)}}^{-1}u$] {};
			\draw (Yl)
				to node[left]{$i_{j_1-1}$}
				+(-.5,-1) node[vertex] (l3) {};
			\draw (Yl)
				to node[right]{$i_{j_1}$}
				+(0.3,-1) node[vertex] {};
			\draw [thick,dotted]
				(l3) -- +(0,-1.5) node [solid, vertex] (l4) {};
			\draw
				(l4)
				to node[left]{$i_1$}
				+(0,-1) node[vertex] (l5) [label=below:$d_B^{-1}u$] {};
    	}
	\end{figure}

	\begin{proof}
		We carry out induction on $l=|A\sm B|$, with trivial base case $l=0$.
		Assume $l>0$.
		From Lemma \ref{theo:XA_XA'}(5), we have
		$u \ge_L d^{-1}_{B}u = s_{i_1}\dots s_{i_{j_{1}-1}} d^{-1}_{B^{(1)}} u$,
		and hence $d^{-1}_{B^{(1)}}u \ge_L s_{i_1}\dots s_{i_{j_{1}-1}} d^{-1}_{B^{(1)}} u$ by Lemma \ref{theo:weakorder_xyz}.
		Hence
		\begin{align}
			u 
				&= \phi_{d_A}(d^{-1}_{B}u) \notag \\
				&= \phi_{d_{A^{(1)}}}\phi_{i_{j_1}}\phi_{i_{j_1-1}}\dots\phi_{i_1}(s_{i_1}\dots s_{i_{j_{1}-1}} d^{-1}_{B^{(1)}} u ) \notag \\
				&= \phi_{d_{A^{(1)}}}\phi_{i_{j_1}}(d^{-1}_{B^{(1)}} u). \label{eq:phiA_dBu}
		\end{align}
		
		We now claim $s_{i_{j_1}} d^{-1}_{B^{(1)}} u < d^{-1}_{B^{(1)}} u$;
		suppose to the contrary that 
		$s_{i_{j_1}} d^{-1}_{B^{(1)}} u > d^{-1}_{B^{(1)}} u$.
		Then we have
		$s_{i_{j_1}} d^{-1}_{B^{(1)}} u \not\le_L u$ since
		$l(s_{i_{j_1}} d^{-1}_{B^{(1)}} u) >l(u) - l(s_{i_{j_1}} d^{-1}_{B^{(1)}})$.
		On the other hand,
		$s_{i_{j_1}} d^{-1}_{B^{(1)}} u > d^{-1}_{B^{(1)}} u$ implies
		$\phi_{i_{j_1}}(d^{-1}_{B^{(1)}} u) = s_{i_{j_1}} d^{-1}_{B^{(1)}} u$,
		which implies
		$\phi_{d_{A^{(1)}}}(s_{i_{j_1}} d^{-1}_{B^{(1)}} u) = u$ by (\ref{eq:phiA_dBu}),
		which implies
		$s_{i_{j_1}} d^{-1}_{B^{(1)}} u \le_L u$ by Lemma \ref{theo:b_xyz},
		which is in contradiction.
		
		Therefore $s_{i_{j_1}} d^{-1}_{B^{(1)}} u < d^{-1}_{B^{(1)}} u$, that is,
		$\phi_{i_{j_1}}(d^{-1}_{B^{(1)}} u)=d^{-1}_{B^{(1)}} u$,
		and hence $\phi_{d_{A^{(1)}}}(d^{-1}_{B^{(1)}} u)=u$ by (\ref{eq:phiA_dBu}).
		Hence, 
		since $|A^{(1)}\sm B^{(1)}|=|A\sm B|-1$,
		we obtain $s_{i_{j_a}} d^{-1}_{B^{(a)}} u < d^{-1}_{B^{(a)}} u$
		for $a=2,\dots,l$
		by the induction hypothesis applied for $(A,B):=(A^{(1)},B^{(1)})$.
	\end{proof}
	
	\begin{lemm}\label{theo:XA_booleanlattice}
		Let $A,B\in Z'_{u,-}$ with $B\subset A$.
		The following are equivalent:
		\begin{enumerate}[label=\textup{(\arabic*)}]
			\item
				$B\in X'_{A,u}$.
			\item
				$B\cup\{i\}\in Z'_{u,-}$ for any $i\in A\sm B$.
			\item
				$B\cup\{i\}\in X'_{A,u}$ for any $i\in A\sm B$.
			\item
				$A\sm\{i\}\in Z'_{u,-}$ for any $i\in A\sm B$.
			\item
				$A\sm\{i\}\in X'_{A,u}$ for any $i\in A\sm B$.
			\item
				$[B,A]\subset Z'_{u,-}$.
			\item
				$[B,A]\subset X'_{A,u}$.
		\end{enumerate}
	\end{lemm}
	\Todo{relation to the M{\"o}bius function $\varphi_{Z'_{u,-}}$?}
	\begin{proof}
		\noindent\underline{$(2)\iff(4)\iff(6)$}:
		$(6)\implies(4)$ and $(6)\implies(2)$ are obvious.
		$(2)\implies(4)\implies(6)$ is from Lemma \ref{theo:Zu_props}(1).
	
		\noindent\underline{$(1)\implies (2)$}:
		We use the notations $A^{(a)}$ and $B^{(a)}$ in Lemma \ref{theo:XA_props}.
		From Lemma \ref{theo:XA_props} we have
		$\{i_{j_a}\}\cup B^{(a)} \in Z'_{u,-}$ for any $a$,
		and hence $B\cup\{i_{j_a}\} = (\{i_{j_a}\}\cup B^{(a)})\cup B \in Z'_{u,-}$ 
		by Proposition \ref{theo:Zu_props}(1).

		\noindent\underline{$(1)\implies(7)$}:
		We already proved $(1)\implies(2)\iff(6)$.
		Hence, since $A,B\in X'_{A,u}$ and $[B,A]\subset Z'_{u,-}$,
		by Lemma \ref{theo:XA_convex} \Todo{why} we have $[B,A]\subset X'_{A,u}$.
		
		\noindent\underline{$(1)\iff(3)\iff(5)\iff(7)$}:
		It is obvious that $(7)\implies (3),(5)$.
		From Lemma \ref{theo:XA_cap}(1) we have $(3)\implies(1)$ and $(5)\implies(1)$.
		Besides we already proved $(1)\implies(7)$.
		
		\noindent\underline{$(4)\implies(5)$}:
		By Lemma \ref{theo:XA_cover}.
	\end{proof}
	
We write $\bigcap X = \bigcap_{x\in X} x$ for a set $X$ of sets.
\begin{coro}\label{theo:YA_boolean}
	We have $X'_{A,u}=[\bigcap X'_{A,u}, A]$ if $X'_{A,u}\neq\emptyset$,
	and $Y'_{A,u}=[\bigcap Y'_{A,u}, A]$ if $Y'_{A,u}\neq\emptyset$.
	In particular, $X'_{A,u}$ and $Y'_{A,u}$ are isomorphic to boolean posets,
	and therefore so are $X_{A,u}$ and $Y_{A,u}$.
\end{coro}
\begin{proof}
	
	Assume $Y'_{A,u}$ is nonempty.
	Then $Y'_{A,u}$ has 
	the minimum element $C=\bigcap Y'_{A,u}$ by Lemma \ref{theo:XA_cap}(2).
	By Lemma \ref{theo:XA_booleanlattice}(1)$\implies$(7) we have $[C,A]\subset X'_{A,u}$.
	Moreover, 
	since $Y_{A,u}$ is an order ideal of $X_{A,u}$
	we have $Y'_{A,u}$ is an order filter of $X'_{A,u}$,
	and therefore $[C,A]\subset Y'_{A,u}$.
	The opposite inclusion $Y'_{A,u}\subset[C,A]$ is implied by minimality of $C$.
	Therefore $Y'_{A,u} = [C,A]$.
	
	It is proved similarly that $X'_{A,u}=[\bigcap X'_{A,u}, A]$ whenever $X'_{A,u}\neq\emptyset$.
\end{proof}
	
	Therefore we have
	\begin{align}\label{eq:sum_over_YA}
		\sum_{v\in Y_{A,u}} (-1)^{|A|-(l(u)-l(v))} 
		&= \sum_{B\in Y'_{A,u}} (-1)^{|A|-(l(u)-l(d_B^{-1}u))} \\ 
		&= \sum_{B\in Y'_{A,u}} (-1)^{|A|-|B|} \notag \\ 
		&=
		\begin{cases}
			1 & \text{if ($|Y_{A,u}|=$) $|Y'_{A,u}|=1$}, \\
			0 & \text{otherwise}.
		\end{cases} \notag
	\end{align}
	
\subsubsection{Step 4}

	Next we discuss which $A$ satisfies the condition $|Y_{A,u}|=1$.

	Since $Z_{u,-}\subset [e,u]_L$ is an order filter,
	so is $Z_{u,-}\cap [e,w]\subset [e,u]_L \cap [e,w]$.
	Hence, if ($w\SLmeet u =$) $\max([e,u]_L \cap [e,w]) \notin Z_{u,-}$,
	then $Z_{u,-}\cap [e,w] = \emptyset$,
	and hence $Y_{A,u}=\emptyset$ for any $A$
	since $Y_{A,u} = X_{A,u}\cap[e,w] \subset Z_{u,-}\cap [e,w]$.
	We hence assume $w \SLmeet u \in Z_{u,-}$
	and write $w \SLmeet u = d_{A_0}^{-1}u$ with $A_0\in Z'_{u,-}$.
	Write $Z_{u,-}^{\le w}=Z_{u,-}\cap[e,w]$.
	Note that $w\SLmeet u = \max Z_{u,-}^{\le w}$.

\begin{exam}
	Recall Example \ref{exam:XA}.
	In that case $\mathrm{max}(Z_{u,-}\cap[e,w]) = s_1 u$ and hence $A_0=\{1\}$.
	It is easily checked that
	$X_{\{1\},u}=\{u,s_1 u\}$ and
	$Y_{\{1\},u}=\{s_1 u\}$.
\end{exam}

\begin{lemm}\label{theo:YA=1}
	$|Y_{A,u}| = 1 \iff A = A_0$.
\end{lemm}
\begin{proof}
	($\implies$)
	Clearly $d_{A_0}^{-1}u \in Y_{A_0,u}$.
	On the contrary, take any $v\in Y_{A_0,u}$.
	Then $v=d_B^{-1}u$ for some $B\in Y'_{A_0,u}$.
	Since $Y'_{A_0,u}\subset X'_{A_0,u} \subset [\emptyset, A_0]$,
	we have $B\subset A_0$.
	On the other hand,
	since $v\in Y_{A_0,u} = X_{A_0,u}\cap[e,w]\subset Z_{u,-}^{\le w}$,
	we have $v\le \max Z_{u,-}^{\le w} = d_{A_0}^{-1}u$,
	and hence $B\supset A_0$. Therefore $B=A_0$.

	($\impliedby$)
	If $A\notin Z'_{u,-}$, % (i.e.\ $d_A^{-1}u\not\le_L u$),
	then $|Y_{A,u}| \le |X_{A,u}| = 0$ from Lemma \ref{theo:XA_XA'}(1).
	We hence assume $A\in Z'_{u,-}$.
	Then $d_A^{-1}u\in Z_{u,-}$.
	
	If $d_A^{-1}u \not\le w$, 
	then $Y_{A,u} = \emptyset$ since $d_A^{-1}u$ is the minimum element of $X_{A,u}$
	and $Y_{A,u}=X_{A,u}\cap[e,w]$ is an order ideal of $X_{A,u}$.
	
	Hence we assume $d_A^{-1}u \le w$.
	Since $d_{A_0}^{-1}u = \max Z_{u,-}^{\le w}$,
	we have $d_A^{-1}u\le d_{A_0}^{-1}u$, and hence $A_0\subset A$.
	Suppose $A_0\subsetneq A$.
	By Corollary \ref{theo:CP_Z'+-}
	there exists an $A'\in Z'_{u,-}$ such that $A_0\subset A'\subset A$ and $|A\sm A'|=1$.
	By Lemma \ref{theo:XA_cover} and that $d_{A'}^{-1}u\le d_{A_0}^{-1}u\le w$
	we have $d_{A'}^{-1}u\in Y_{A,u}$.
	Hence $Y_{A,u}\supset \{d_{A}^{-1}u, d_{A'}^{-1}u\}$.
\end{proof}	

Therefore, substituting (\ref{eq:sum_over_YA}) and the result of Lemma \ref{theo:YA=1} into the right-hand side of (\ref{eq:gu_coeff_sumsplit}) and noting that $|A_0|=l(u)-l(w\SLmeet u)$, we have
$$
	[\kks{u}](\kkss{w} \wth_r) =
		\begin{cases}
			1 & \text{if $w\SLmeet u\in Z_{u,-}$ and $l(u)-l(w\SLmeet u) \le r$,} \\
			0 & \text{otherwise.}
		\end{cases}
$$
Finally, we show the following:
\begin{lemm}\label{theo:A0_r}
	The following are equivalent:
	\begin{enumerate}[label=\textup{(\arabic*)}]
		\item
			$w\SLmeet u\in Z_{u,-}$ and $l(u)-l(w\SLmeet u) \le r$.
		\item
			There exists $A$ such that $|A|\le r$ and $u\ge_L d_A^{-1}u \le w$.
		\item
			There exists $A$ such that $|A|\le r$ and $u\le d_A w \ge_L w$.
		\item
			There exists $A$ such that $|A|=r$ and $u\le d_A w \ge_L w$.
	\end{enumerate}
\end{lemm}
\begin{proof}
	\noindent (1)$\iff$(2):
	Clear.

	\noindent (3)$\iff$(4):
	$(4) \implies (3)$ is obvious.
	$(3) \implies (4)$ follows from the fact $Z'_{u,+}$ has the Chain Property and the maximum element of size $k$, which corresponds to the maximum element of $Z_{u,+}$.
	
	\noindent (2) $\implies$ (3):
	Assume $u\ge_L d_A^{-1} u \le w$.
	Then $u=\phi_{d_A}(d_A^{-1} u) \le \phi_{d_A}(w)$ by Lemma \ref{theo:b_i}(2).
	Besides, we have $\phi_{d_A}(w) = d_{B}w \ge_L w$ for some $B\subset A$ by Lemma \ref{theo:b_xyz},
	and $|B|\le |A| \le r$.
	
	\noindent (3) $\implies$ (2):
	Proved similarly to (2) $\implies$ (3), with Lemma \ref{theo:psi_xyz} instead of Lemma \ref{theo:b_xyz}.
\end{proof}

Now we finished, from Lemma \ref{theo:A0_r} (1)$\iff$(4),
the proof of Theorem \ref{theo:StrongSumPieri}:
$$
	\kkss{w} \wth_{r} = \sum_{u} \kks{u},
$$
summed over $u\in \Wo$ such that $u \le d_Aw$
for some $A\subsetneq I$ with $|A|=r$ and $d_Aw\ge_L w$.

Theorem \ref{theo:StrongSumPieri'} follows from
Theorem \ref{theo:StrongSumPieri},
Corollary \ref{theo:meetSemiLattice}, and
the Inclusion-Exclusion Principle.

\section{Proof of the $k$-rectangle factorization formula}\label{sect:StrongSumFactorization}

This section is devoted for the proof of Theorem \ref{theo:StrongSumFactorization}.

The idea of the proof is similar to that of Proposition \ref{prop:kS_fac};
	we consider a linear map $\Theta: \La^{(k)} \longrightarrow \La^{(k)}$ extending 
	$\kkss{\la} \mapsto \kkss{\Rt\cup\la}$,
	having that $\{\kkss{\la}\}_{\la\in\Pk}$ forms a basis of $\Lk$.
	It suffices to show $\Theta$ is a $\Lk$-homomorphism,
	since it implies $\kkss{\Rt\cup\la} = \Theta(\kkss{\la}) = \kkss{\la} \Theta(1) = \kkss{\la} \Theta(\kkss{\emptyset}) = \kkss{\la} \kkss{\Rt}$.
	Since $\{\wth_i\}_{1\le i \le k}$ generate $\Lk$,
	we only need to show 
	\begin{equation}\label{eq:Theta_hr}
		\Theta(\wth_r \kkss{\la}) = \wth_r \Theta(\kkss{\la}).
	\end{equation}
    Let $d_{A_1}\la, d_{A_2}\la, \dots$ 
    be the list of all weak strips over $\la$ of size $r$.
	Applying Theorem \ref{theo:StrongSumPieri'} to both sides of (\ref{eq:Theta_hr}), 
	we have
	\begin{align}
		\text{(LHS)}
    		&= \Theta\left(
				\sum_{a} \kkss{d_{A_a}\la}
    			- \sum_{a<b} \kkss{d_{A_{a}\cap A_b}\la}	
				+ \dots
				\right) \notag \\
    		&= \sum_{a} \kkss{R_t\cup (d_{A_a}\la)}
   				- \sum_{a<b} \kkss{R_t\cup (d_{A_{a}\cap A_b}\la)}
 				+ \dots, 
				\label{eq:FactorizationLHS}
\intertext{and by Lemma \ref{theo:Rt} (3) we have}
		\text{(RHS)}
			&= \wth_r \kkss{R_t\cup \la} \notag \\
			&= \sum_{a} \kkss{d_{A_a+t}(R_t\cup \la)}
				- \sum_{a<b} \kkss{d_{(A_{a}+t)\cap (A_b+t)}(R_t\cup \la)}
				+ \dots. 
				\label{eq:FactorizationRHS}
	\end{align}
Since $(A_{a}+t)\cap (A_b+t) \cap\dots = (A_{a}\cap A_b\cap\cdots)+t$,
by Lemma \ref{theo:Rt} (1) we have
$(\ref{eq:FactorizationLHS})=(\ref{eq:FactorizationRHS})$.

Now Theorem \ref{theo:StrongSumFactorization} is proved.


\end{document}